\documentclass[12pt,a4paper,twoside]{article}
\usepackage[ansinew]{inputenc}
\usepackage{amsmath}
\usepackage{amsfonts}
\usepackage{amssymb}
\usepackage{amsthm}
\usepackage{tabu}
\usepackage{tikz}
\usepackage{ifthen}
 \usetikzlibrary{shapes.geometric,arrows,snakes,backgrounds}

\makeatletter
\def\@seccntformat#1{\csname the#1\endcsname.\quad}
\makeatother

\newcommand{\OJO}{\textbf{???}\ }
\def\R{\mathbb{R}}
\def\0{\mathbf{0}}
\def\ii{\mathbf{i}}
\def\RR{\mathbb{R}^2}
\def\SS{\mathbb{S}^2}
\newcommand{\diam}{\mathop{\rm diam }\nolimits}
\newcommand{\Bd}{\mathop{\rm Bd}\nolimits}
\newcommand{\Inte}{\mathop{\rm Int}\nolimits}
\newcommand{\Cl}{\mathop{\rm Cl}\nolimits}

\definecolor{beren}{rgb}{0.4,0.04,0.28}
\definecolor{rojo}{rgb}{0.5,0,0}
\definecolor{verde}{rgb}{0,0.4,0}
\definecolor{granate}{RGB}{130,24,130}
\definecolor{azuloscuro}{rgb}{0,0,0.8}
\definecolor{naranja}{RGB}{255,96,0}
\definecolor{marron}{rgb}{0,0.4,0}

\newcommand{\Fizq}{\LARGE $\boldsymbol{\leftleftarrows}$}
\newcommand{\Fder}{\LARGE $\boldsymbol{\rightrightarrows}$}
\newcommand{\Farr}{\LARGE $\boldsymbol{\upuparrows}$}
\newcommand{\Faba}{\LARGE $\boldsymbol{\downdownarrows}$}
\newcommand{\fizq}{\large $\boldsymbol{\leftleftarrows}$}
\newcommand{\fder}{\large $\boldsymbol{\rightrightarrows}$}
\newcommand{\farr}{\large $\boldsymbol{\upuparrows}$}
\newcommand{\faba}{\large $\boldsymbol{\downdownarrows}$}

\newcommand{\SWNE}{$\boldsymbol{\nearrow}$}
\newcommand{\NWSE}{$\boldsymbol{\searrow}$}
\newcommand{\SENW}{$\boldsymbol{\nwarrow}$}
\newcommand{\NESW}{$\boldsymbol{\swarrow}$}

\newcommand{\swne}{\tiny $\boldsymbol{\nearrow}$}
\newcommand{\nwse}{\tiny $\boldsymbol{\searrow}$}
\newcommand{\senw}{\tiny $\boldsymbol{\nwarrow}$}
\newcommand{\nesw}{\tiny $\boldsymbol{\swarrow}$}

\newcounter{letra}
\renewcommand{\theletra}{\Alph{letra}}
\newtheorem{theorem}{Theorem}[section]
\newtheorem{proposition}[theorem]{Proposition}

\newtheorem{lemma}[theorem]{Lemma}
\newtheorem{maintheorem}[letra]{Theorem}

\theoremstyle{definition}
\newtheorem{definition}[theorem]{Definition}
\theoremstyle{remark}
\newtheorem{remark}[theorem]{Remark}

\theoremstyle{plain}

\newcounter{versionfinal}
\ifthenelse{\value{versionfinal}=1}{
\newcommand{\josegines}[1]{}

\newcommand{\corregidooriginal}[2]{
}
\newcommand{\borrar}[1]{
}
\newcommand{\morado}[1]{\textcolor{black}{#1}}
\newcommand{\moradomucho}{\color{black}}
\newcommand{\negromucho}{\color{black}}
\newcommand{\apendices}[1]{
}
\renewcommand{\poneralfinal}{}
}{

\newcommand{\josegines}[1]{\renewcommand{\thefootnote}{\bfseries\color{red}\arabic{footnote}}\footnote{\textcolor{red}{\textbf{Nota de Jose Gines: } #1}}\renewcommand{\thefootnote}{\arabic{footnote}}}
\newcommand{\victor}[1]{\renewcommand{\thefootnote}{\bfseries\color{blue!50!black}\arabic{footnote}}\footnote{\textcolor{blue!50!black}{\textbf{Nota de Victor: } #1}}\renewcommand{\thefootnote}{\arabic{footnote}}}

\newcommand{\corregidooriginal}[2]{#1

{\bfseries \color{red} #2}

}

\newcommand{\borrar}[1]{
}

\newcommand{\morado}[1]{\textcolor{purple}{#1}}
\newcommand{\azul}[1]{\textcolor{blue}{#1}}

\newcommand{\moradomucho}{\color{purple}}
\newcommand{\negromucho}{\color{black}}

\newcommand{\apendices}[1]{#1}

}

\begin{document}

\title{A topological classification of plane polynomial systems
having a globally attracting singular point}
\author{Jos\'e Gin\'es Esp\'in Buend\'ia and V\'ictor Jim\'enez
Lop\'ez}
\date{\normalsize{Universidad de Murcia (Spain)}\\
\normalsize{\today}}
\maketitle

\begin{abstract}
In this paper, plane polynomial systems having a singular point
attracting all orbits in positive time are classified up to
topological equivalence. This is done by assigning a combinatorial
invariant to the system (a so-called ``feasible set'' consisting
of finitely many vectors with components in the set $\{n/3:
n=0,1,2,\ldots\}$), so that two such systems are equivalent if and
only if (after appropriately fixing an orientation in
$\mathbb{R}^2$ and a heteroclinic separatrix) they have the same
feasible set. In fact, this classification is achieved in the more
general setting of continuous flows having finitely many
separatrices.

Polynomial representatives for each equivalence class are found,
although in a non-constructive way. Since, to the best of our
knowledge, the literature does not provide any concrete polynomial
system having a non-trivial globally attracting singular point, an
explicit example is given as well.
\end{abstract}


\smallbreak \noindent{\bf 2010 Mathematics Subject
Classification:} 34C37, 37C10, 37C15.

\section{Introduction and statements of the main results}

Classifying the phase portraits of plane polynomial systems, that
is, those of the form
 $$
 \begin{cases}
 x'=P(x,y), \\
 y'=Q(x,y),
 \end{cases}
 $$
with $P(x,y)$ and $Q(x,y)$ polynomials in the variables $x$ and
$y$, is a classical problem (many would say the problem \emph{par
excellence}) of the qualitative theory of differential equations.
As a whole it is a daunting, probably insurmountable, task, which
if completed would provide, as a by-product, an answer to the
famous (second part of the) Hilbert 16th problem asking for a
bound $H(n)$ on the number of limit cycles of the system in terms
of the maximum degree $n$ of $P(x,y)$ and $Q(x,y)$. Presently,
this bound is unknown even in the quadratic case $n=2$; in fact,
although there are strong reasons to conjecture $H(2)=4$, not even
the finiteness of $H(2)$ has been established.

Understandably, researchers in this area have added dynamical
and/or analytic restrictions to the problem, as in \cite{RSv},
where it is shown that any $C^1$-structurally stable system with
finitely many singular points and limit cycles is topologically
equivalent to a polynomial system, or as in \cite{BD}, where
complex polynomial systems (that is, polynomial systems such that
$P(x,y)$ and $Q(x,y)$ satisfy the Cauchy-Riemann conditions) are
fully described in terms of appropriate combinatorial and analytic
data.

No doubt fuelled by the search of a proof for the elusive equality
$H(2)=4$, quadratic systems have got the lion's share of this
work. Here, among many others, cordal \cite{GLL}, Lotka-Volterra
\cite{SV}, those having a center \cite{Vu}, homogeneous \cite{Da},
Hamiltonian \cite{AL}, and bounded systems \cite{DP} have been
classified up to topological equivalence. The monograph \cite{Re}
is specifically devoted to this subject; interestingly, in p.~303
there, the number of possible portrait phases for quadratic
systems (under the hypothesis $H(2)=4$) is estimated to be around
2000.

Somewhat surprisingly, the most natural problem of classifying
polynomial systems with a globally attracting singular point, that
is, those whose orbits tend in positive time to the same singular
point (which we can assume, without loss of generality, to be the
origin $\0$), has not been studied yet. A possible explanation for
this is that such a classification is pretty trivial in the
quadratic realm: these systems are equivalent to the linear
attracting node $x'=-x, y'=-y$. The reason is the following. As we
will see below (Remark~\ref{ell-hyp}), in order to avoid the above
trivial case, the finite sectorial decomposition at $\0$ must
include both an elliptic and a hyperbolic sector. Such a local
behaviour is certainly possible for quadratic systems: an explicit
example with an \emph{elliptic saddle} (that is, a decomposition
consisting exactly of one elliptic sector and one hyperbolic
sector) can be found in \cite[p. 368]{ALGM}. Nevertheless, global
attraction implies that the system is bounded (that is, it has
bounded positive semi-orbits), and for these systems the existence
of elliptic sectors at singular points is excluded by \cite{DP}.
Incidentally, if a $C^1$-system is locally holomorphic at $\0$,
that is, the Cauchy-Riemann conditions hold near $\0$, then either
$\0$ is a topological node or the sectorial decomposition consists
of exactly evenly many elliptic sectors \cite{BT}. Therefore,
non-trivial global attraction is also impossible in this case.

In the present paper we fulfil this gap by classifying polynomial
global attraction up to topological equivalence. Indeed we work in
the much more general setting of (continuous) flows with finitely
many separatrices (or equivalently, see Remark~\ref{daigual},
those having the finite sectorial decomposition property at $\0$,
or those having finitely many unstable orbits), when their
separatrix skeletons (the union of all separatrices and exactly
one orbit from each region in the complementary set) are also
finite. To begin with, there is a dichotomy: global attraction is
trivial if and only if $\0$ is positively stable, that is, there
are no regular homoclinic orbits
(Proposition~\ref{atractorestable}(i)). Hence we concentrate in
what follows in the ``non-positively stable'' case, when at least
(as implied by Proposition~\ref{atractorestable}(ii)) one
heteroclinic separatrix must exist. We rely on a well-known result
by Markus \cite{Ma}, later extended by Neumann \cite{Ne} (see also
\cite{EJ}), stating that two flows are equivalent if and only if
there is a plane homeomorphism preserving the orbits and time
directions of their separatrix skeletons (Theorem~\ref{markus}).
As it turns out, a weaker so-called \emph{compatibility} condition
(just assuming preservation of orbits, see
Subsection~\ref{extensiones}) suffices, provided that at least one
heteroclinic separatrix is preserved as well. Moreover, after
fixing an orientation in $\RR$ (counterclockwise or clockwise) and
a heteroclinic separatrix, and using the skeleton combinatorial
structure, there is a canonical way to associate  a so-called
feasible set (a finite vectorial set as described in
Definition~\ref{feasible}) to the flow, and this labelling
characterizes equivalence: topologically equivalent flows have the
same canonical feasible set. We emphasize that although the
separatrix skeleton is not uniquely defined, no ambiguity arises
because the corresponding canonical feasible sets are the same
(this follows from Lemma~\ref{controlado}).

Our first theorem summarizes these results.

\begin{maintheorem}
 \label{principal}
 Assume that $\0$ is a global
 attractor, non-positively stable, for two flows $\Phi$ and
 $\Phi'$, both having finitely many separatrices, and let
 $\mathcal{X}$ and $\mathcal{X'}$ denote
 their separatrix skeletons. Then the following
 statements are equivalent:
 \begin{itemize}
  \item[(i)] $\Phi$ and $\Phi'$ are topologically equivalent.
  \item[(ii)] $\mathcal{X}$ and $\mathcal{X'}$ are compatible
   and  the compatibility bijection
   $\xi:\mathcal{X}\to \mathcal{X'}$ maps some
   heteroclinic separatrix of $\Phi$ to a heteroclinic
   separatrix of $\Phi'$.
  \item[(iii)] There are respective
   orientations $\Theta,\Theta'$ in $\RR$ and heteroclinic
   separatrices $\Sigma,\Sigma'$
   such that the associated canonical feasible sets are the same.
 \end{itemize}
\end{maintheorem}

Contrary to what one might initially expect, the index of the
global attractor plays no role in this topological classification.
In fact, after extending the flow to the Riemann sphere, we get
that $\infty$ is a repelling (topological) node
(Remark~\ref{fundamental}). Hence, its index is 1 and, by the
Poincar\'e-Hopf theorem \cite[p. 179]{DLA}, the index of the
attractor is 1 as well. Moreover, sharing (up to homeomorphisms)
the same finite sectorial decomposition is a necessary but not
sufficient condition for two such flows being topologically
equivalent, see Figure~\ref{FSDnobasta}. Likewise, compatibility
alone is not enough to guarantee topological equivalence, see
Figure~\ref{compatiblenobasta}.

\begin{figure}
\centering
\begin{tikzpicture}[scale=.15]


\begin{scope}[xshift=-28cm]
\draw[line width=0.75, color=gray]
   (3.0,-18.0) .. controls +(90:4) and +(225:5) ..  (11.0,-5.0)
                     .. controls +(45:4) and +(270:5) ..  (16.0, 8.0)
                     .. controls +(90:5) and +(0:7) ..  (2.0,29.0)
                     .. controls +(180:7) and +(80:4) ..  (-16.0,10.0)
                     .. controls +(260:4) and +(120:4) ..  (-13.0,0.0)
                     .. controls +(300:4) and +(215:5) ..  (-4.0,-4.0)
                     .. controls +(35:3) and +(225:3) ..  (0.0,0.0);

\draw[->,line width=1.0, color=gray]
    (11.0,-5.0) .. controls +(45:0.05) and +(225:0.05) ..  (11.0,-5.0);

\draw[line width=1.00]
   (-0.1,-20.0) .. controls +(90:1) and +(280:1) .. (0.1,-10.0)
           .. controls +(100:1) and +(270:0.5) ..  (0,0);
\draw[->, line width=1.00]
    (0.1,-10.0) .. controls +(100:0.05) and +(280:0.05) ..  (0.1,-10.0);

\draw[line width=1.00]
   (0.0,0.0) .. controls +(0:10) and +(260:4) .. (13.0,11.0)
           .. controls +(80:4) and +(0:4) ..  (0,25.0)
                     .. controls +(180:4) and +(80:10) ..  (-11,9.0)
                    .. controls +(260:10) and +(180:10) ..  (0,0);
\draw[->,line width=1.00]
    (-11,9.0) .. controls +(260:0.05) and +(80:0.05) .. (-11,9.0);

\draw[line width=1.00]
   (0.0,0.0) .. controls +(60:2) and +(270:2) .. (3.0,6.0)
           .. controls +(90:2) and +(0:3) ..  (-1.0,13.0)
                     .. controls +(180:3) and +(90:4) ..  (-4,7.5)
                    .. controls +(270:4) and +(110:4) ..  (0,0);
\draw[->, line width=1.00]
    (3.0,6.0) .. controls +(270:0.05) and +(90:0.05) .. (3.0,6.0);

\draw[line width=0.75, color=gray]
   (0.0,0.0) .. controls +(80:2) and +(265:2) .. (0.7,3.0)
           .. controls +(85:2) and +(0:3) ..  (-0.5,10.0)
                     .. controls +(180:3) and +(90:2) ..  (-2.0,5.5)
                    .. controls +(270:2) and +(95:4) ..  (0,0);
\draw[->,line width=1.0, color=gray]
    (0.7,3.0) .. controls +(270:0.05) and +(90:0.05) .. (0.7,3.0);

\draw[line width=0.75, color=gray]
   (0.0,0.0) .. controls +(20:2) and +(250:3) .. (9.5,6.0)
               .. controls +(70:3) and +(-60:3) ..  (9.3,14.0)
                       .. controls +(120:3) and +(-45:3) ..  (4.5,20.0)
                       .. controls +(135:3) and +(50:3) ..  (-5.0,18.0)
             .. controls +(230:2) and +(80:2) ..  (-9.0,11.0)
                       .. controls +(260:2) and +(155:3) ..  (-8.5,4.0)
                       .. controls +(-25:3) and +(225:3) ..  (-4.5,13.0)
                       .. controls +(45:3) and +(180:2) ..(0.0,15.0)
                       .. controls +(0:4) and +(60:3) ..  (3.0,3.0)
                       .. controls +(240:3) and +(40:2) ..  (0,0);

\draw[->,line width=1.0, color=gray]
    (-5,18.0) .. controls +(230:0.05) and +(50:0.05) .. (-5,18.0);

\fill[color=gray!110!black] (0,0) circle (0.5);
\end{scope}

\begin{scope}[xshift=28cm]

\draw[line width=1.0, color=gray]
   (-0.1,-20.0) .. controls +(90:1) and +(280:1) .. (0.1,-10.0)
           .. controls +(100:1) and +(270:0.5) ..  (0,0);
\draw[->, line width=1.0, color=gray]
    (0.1,-10.0) .. controls +(100:0.05) and +(280:0.05) ..  (0.1,-10.0);

\draw[line width=1.0]
   (8,-17.0) .. controls +(110:1) and +(290:2) .. (3,-8.0)
           .. controls +(110:2) and +(330:0.5) ..  (0,0);
\draw[->,line width=1.0]
    (3,-8.0) .. controls +(110:0.05) and +(290:0.05) ..  (3,-8.0);

\draw[line width=1.0]
   (-10,-18.0) .. controls +(70:1) and +(260:2) .. (-5,-9.0)
           .. controls +(80:2) and +(275:0.5) ..  (0,0);
\draw[->,line width=1.0]
    (-5,-9.0) .. controls +(80:0.05) and +(260:0.05) ..  (-5,-9.0);

\draw[line width=1.0, color=gray]
   (13.0,-17.5) .. controls +(110:3) and +(280:4) .. (10.0,-10.0)
           .. controls +(100:4) and +(270:5) ..  (15,0)
                     .. controls +(90:4) and +(0:5) ..  (8,30)
                     .. controls +(180:5) and +(90:5) ..  (0,0);
\draw[->, line width=1.0, color=gray]
    (15,0) .. controls +(90:0.05) and +(270:0.05) ..  (15,0);

\draw[line width=1.0, color=gray]
   (-14.0,-18.2) .. controls +(85:3) and +(260:4) .. (-11.0,-9.0)
           .. controls +(80:4) and +(270:5) ..  (-14,1)
                     .. controls +(90:4) and +(180:5) ..  (-9,29)
                     .. controls +(0:5) and +(120:5) ..  (0,0);
\draw[->, line width=1.0, color=gray]
    (-14,1) .. controls +(90:0.05) and +(270:0.05) ..  (-14,1);

\draw[line width=1.00]
   (-5, 30.0) .. controls +(280:1) and +(100:4) .. (-2,15.0)
           .. controls +(280:4) and +(270:0.5) ..  (0,0);
\draw[->, line width=1.00]
    (-2,15.0) .. controls +(280:0.05) and +(100:0.05) ..  (-2,15.0);

\draw[line width=1.00]
   (0, 0.0) .. controls +(315:1) and +(180:4) .. (7.0,-3.0)
            .. controls +(0:4) and +(270:2.5) ..  (12,3.0)
                        .. controls +(90:4) and +(0:5) ..  (8,15.0)
                        .. controls +(180:4) and +(45:0.5) ..  (0.0,0.0);
\draw[->, line width=1.00]
    (12,3.0) .. controls +(90:0.05) and +(270:0.05) ..  (12,3.0);

\draw[line width=1.0, color=gray]
   (0, 0.0) .. controls +(340:2) and +(180:4) .. (5.0,-1.0)
            .. controls +(0:4) and +(270:1.5) ..  (9.0,3.0)
                        .. controls +(90:2) and +(0:2) ..  (5.5,5.0)
                        .. controls +(180:2) and +(15:2.5) ..  (0.0,0.0);
\draw[->, line width=1.0, color=gray]
    (5.0,-1.0) .. controls +(0:0.05) and +(180:0.05) ..  (5.0,-1.0);

\draw[line width=1.00]
   (0, 0.0) .. controls +(210:3) and +(0:4) .. (-6.0,-4.0)
            .. controls +(180:4) and +(270:2.5) ..  (-11,3.5)
                        .. controls +(90:4) and +(180:5) ..  (-8.5,14.5)
                        .. controls +(0:4) and +(135:5) ..  (0.0,0.0);
\draw[->, line width=1.00]
    (-11,3.5) .. controls +(90:0.05) and +(270:0.05) ..  (-11,3.5);

\draw[line width=1.0, color=gray]
   (0, 0.0) .. controls +(190:2) and +(0:4) .. (-5.0,-1.5)
            .. controls +(180:4) and +(270:1.5) ..  (-10.0,2.5)
                        .. controls +(90:2) and +(180:2) ..  (-5.5,6.0)
                        .. controls +(0:2) and +(155:2.5) ..  (0.0,0.0);
\draw[->, line width=1.0, color=gray]
    (-5.0,-1.5) .. controls +(180:0.05) and +(0:0.05) ..  (-5.0,-1.5);

\fill[color=gray!110!black] (0,0) circle (0.5);

\end{scope}

\end{tikzpicture}

\caption{Two non-equivalent phase portraits with the same
sectorial decomposition
(elliptic-elliptic-hyperbolic-attracting-hyperbolic in
counterclockwise sense) at the origin.\label{FSDnobasta}}
\end{figure}
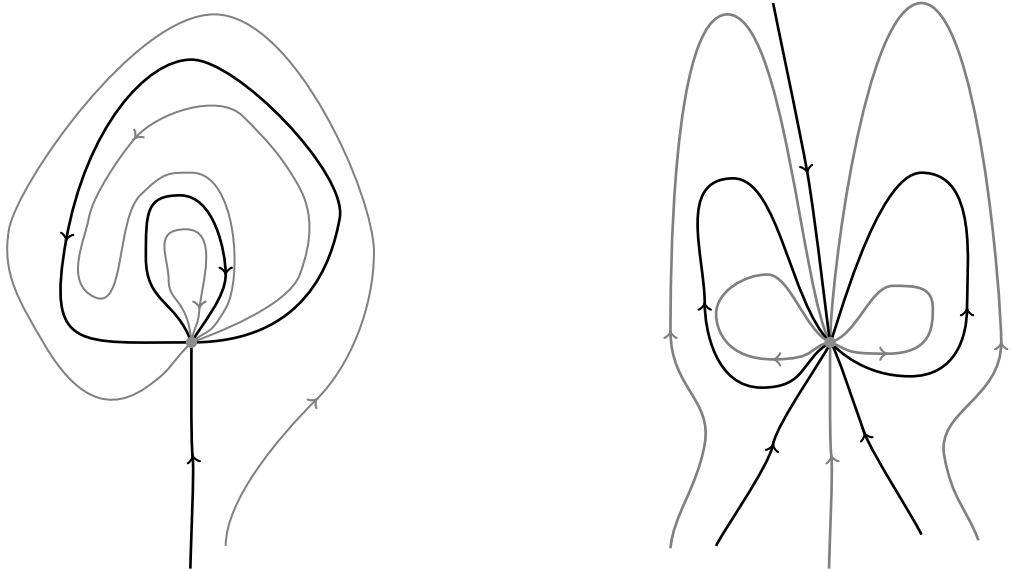

\begin{figure}
\centering

\begin{tikzpicture}[scale=0.21]

\begin{scope}[xshift=0cm]

\draw[line width=1]
   (17,0) .. controls +(180:4) and +(0:4) .. (15,0)
           .. controls +(180:4) and +(0:4) ..  (0,0);
\draw[->, line width=1]
    (15,0) .. controls +(180:0.05) and +(0:0.05) ..  (15,0);

\draw[line width=1]
   (-20,0) .. controls +(0:4) and +(180:4) .. (-15,0)
           .. controls +(0:4) and +(180:4) ..  (0,0);
\draw[->, line width=1]
    (-15,0) .. controls +(0:0.05) and +(180:0.05) ..  (-15,0);

\draw[line width=0.75, color=gray]
   (15.2873, 8.60364) .. controls +(135:2) and +(0:3) .. (7.42021, 12.9)
	                    .. controls +(180:2) and +(70:4) .. (3.42021, 9.39692)
           .. controls +(250:4) and +(80:4) ..  (0,0);
\draw[->, line width=0.75, color=gray]
    (7.42021, 12.9) .. controls +(180:0.05) and +(0:0.05) ..  (7.42021, 12.9);

\draw[line width=1.0]
   (0,0) .. controls +(10:2) and +(210:3) .. (9.65926, 2.58819)
         .. controls +(30:3) and +(-45:4) ..  (12.2873, 8.60364)
				 .. controls +(135:4) and +(60:3) ..  (5.73577, 8.19152)
				 .. controls +(240:3) and +(60:2) ..  (0.0, 0.0);
\draw[->, line width=0.75, color=gray]
   (12.2873, 8.60364) .. controls +(135:0.05) and +(-45:0.05) .. (12.2873, 8.60364);

\draw[line width=0.75, color=gray]
   (0,0) .. controls +(20:2) and +(210:2) .. (9.06308, 4.22618)
         .. controls +(30:2) and +(-45:2) ..  (10.649, 7.45649)
				 .. controls +(135:2) and +(60:2) ..  (6.42788, 7.66044)
				 .. controls +(240:2) and +(50:2) ..  (0.0, 0.0);
\draw[->, line width=0.75, color=gray]
   (10.649, 7.45649) .. controls +(135:0.05) and +(-45:0.05) .. (10.649, 7.45649);

\draw[line width=1.0]
   (0,0) .. controls +(90:2) and +(290:3) .. (-0.871543, 9.96195)
         .. controls +(110:3) and +(45:4) ..  (-6.33925, 13.5946)
				 .. controls +(225:4) and +(120:3) ..  (-7.07105, 7.07108)
				 .. controls +(300:3) and +(135:2) ..  (0.0, 0.0);
\draw[->, line width=1.0]
   (-6.33925, 13.5946) .. controls +(45:0.05) and +(225:0.05) .. (-6.33925, 13.5946);

\draw[line width=0.75, color=gray]
   (0,0) .. controls +(95:2) and +(290:2) .. (-2.58818, 9.65926)
         .. controls +(110:2) and +(45:2) ..  (-5.49402, 11.782)
				 .. controls +(225:2) and +(120:2) ..  (-5.73575, 8.19153)
				 .. controls +(300:2) and +(130:2) ..  (0.0, 0.0);
\draw[->, line width=0.75, color=gray]
  (-5.49402, 11.782) .. controls +(45:0.05) and +(225:0.05) .. (-5.49402, 11.782);

\draw[line width=1.0]
   (0,0) .. controls +(155:2) and +(-20:3) .. (-8.66024, 5.00002)
         .. controls +(160:3) and +(70:2) ..  (-14.0954, 5.13034)
				 .. controls +(250:2) and +(160:3) ..  (-9.84807, 1.73651)
				 .. controls +(340:3) and +(175:2) ..  (0.0, 0.0);
\draw[->, line width=1.0]
   (-14.0954, 5.13034) .. controls +(70:0.05) and +(250:0.05) .. (-14.0954, 5.13034);

\draw[line width=0.75, color=gray]
   (0,0) .. controls +(160:2) and +(-20:3) .. (-9.06307, 4.2262)
         .. controls +(160:3) and +(70:1) ..  (-12.216, 4.44629)
				 .. controls +(250:1) and +(160:1) ..  (-9.65925, 2.58821)
				 .. controls +(340:3) and +(170:2) ..  (0.0, 0.0);
\draw[->, line width=0.75, color=gray]
   (-12.216, 4.44629) .. controls +(70:0.05) and +(250:0.05) .. (-12.216, 4.44629);


\draw[line width=1.0]
   (0,0) .. controls +(185:2) and +(45:3) .. (-9.84808, -1.73645)
         .. controls +(225:3) and +(135:2) ..  (-13.2442, -7.04203)
				 .. controls +(315:2) and +(225:3) ..  (-7.07109, -7.07104)
				 .. controls +(45:3) and +(230:2) ..  (0.0, 0.0);
\draw[->,line width=1.0]
   (-13.2442, -7.04203) .. controls +(315:0.05) and +(135:0.05) .. (-13.2442, -7.04203);

\draw[line width=0.75, color=gray]
   (0,0) .. controls +(190:2) and +(45:1) .. (-9.39694, -3.42017)
         .. controls +(225:1) and +(135:2) ..  (-11.4783, -6.1031)
				 .. controls +(315:2) and +(225:1) ..  (-8.19154, -5.73574)
				 .. controls +(45:1) and +(225:2) ..  (0.0, 0.0);
\draw[->, line width=0.75, color=gray]
   (-11.4783, -6.1031) .. controls +(315:0.05) and +(135:0.05) .. (-11.4783, -6.1031);

\draw[line width=1.0]
   (0,0) .. controls +(245:2) and +(90:3) .. (-5.00003, -8.66024)
         .. controls +(270:3) and +(180:2) ..  (-2.60478, -14.7721)
				 .. controls +(0:2) and +(260:3) ..  (1.73644, -9.84808)
				 .. controls +(80:3) and +(290:2) ..  (0.0, 0.0);
\draw[->,line width=1.0]
   (-2.60478, -14.7721) .. controls +(0:0.05) and +(180:0.05) .. (-2.60478, -14.7721);

\draw[line width=0.75, color=gray]
   (0,0) .. controls +(245:2) and +(90:1) .. (-3.42024, -9.39691)
         .. controls +(270:1) and +(180:2) ..  (-2.25748, -12.8025)
				 .. controls +(0:2) and +(260:1) ..  (-0.0000398038, -10.)
				 .. controls +(80:1) and +(290:2) ..  (0.0, 0.0);
\draw[->, line width=0.75, color=gray]
   (-2.25748, -12.8025) .. controls +(0:0.05) and +(180:0.05) .. (-2.25748, -12.8025);

\draw[line width=1.0]
   (0,0) .. controls +(295:2) and +(130:1) .. (4.22614, -9.0631)
         .. controls +(310:1) and +(225:4) ..  (11.4906, -9.64187)
				 .. controls +(45:4) and +(-30:3) ..  (9.65925, -2.58824)
				 .. controls +(150:3) and +(290:2) ..  (0.0, 0.0);
\draw[->,line width=1.0]
   (11.4906, -9.64187) .. controls +(45:0.05) and +(225:0.05) .. (11.4906, -9.64187);

\draw[line width=0.75, color=gray]
   (0,0) .. controls +(300:2) and +(130:3) .. (6.42784, -7.66047)
         .. controls +(310:3) and +(225:2) ..  (9.95854, -8.35629)
				 .. controls +(45:2) and +(-30:3) ..  (8.66023, -5.00004)
				 .. controls +(150:3) and +(280:2) ..  (0.0, 0.0);
\draw[->, line width=0.75, color=gray]
   (9.95854, -8.35629) .. controls +(45:0.05) and +(225:0.05) .. (9.95854, -8.35629);

\draw[line width=0.75, color=gray]
   (-12.60478, -12.7721) .. controls +(-30:4) and +(180:6) .. (-2.60478, -16.7721)
                         .. controls +(0:4) and +(225:6).. (11.4906, -11.64187)
         .. controls +(45:8) and +(350:6) ..  (9.96194, -0.87161)
				 .. controls +(170:6) and +(-15:6) ..  (0, 0);
\draw[->, line width=0.75, color=gray]
   (11.4906, -11.64187) .. controls +(45:0.05) and +(225:0.05) .. (11.4906, -11.64187);
	
\fill[color=gray!110!black] (0,0) circle (.3);

\draw[dashed, line width=0.4] (0,0) circle [radius=10];

\fill (10.0, 0.0) circle (.2); \node[scale=0.4] at (10.5, 0.5) {$2$};
\fill (9.65926, 2.58819) circle (.2); \node[scale=0.4] at (10.2, 2.0) {$3$};
\fill (9.06308, 4.22618) circle (.2); \node[scale=0.4] at (9.7308, 3.8) {$4$};
\fill (6.42788, 7.66044) circle (.2); \node[scale=0.4] at (7.2788, 7.66044) {$5$};
\fill (5.73577, 8.19152) circle (.2); \node[scale=0.4] at (5.53577, 9.19152) {$6$};
\fill (3.42021, 9.39692) circle (.2); \node[scale=0.4] at (4.12021, 9.65692) {$7$};
\fill (-0.871543, 9.96195) circle (.2); \node[scale=0.4] at (-0.841543, 10.96195) {$8$};
\fill (-2.58818, 9.65926) circle (.2); \node[scale=0.4] at (-2.48818, 10.65926) {$9$};
\fill (-5.73575, 8.19153) circle (.2); \node[scale=0.4] at (-5.73575, 9.19153) {$10$};
\fill (-7.07105, 7.07108) circle (.2); \node[scale=0.4] at (-7.07105, 8.07108) {$11$};
\fill (-8.66024, 5.00002) circle (.2); \node[scale=0.4] at (-8.76024, 5.90002) {$12$};
\fill (-9.06307, 4.2262) circle (.2); \node[scale=0.4] at (-9.8307, 3.8262) {$13$};
\fill (-9.65925, 2.58821) circle (.2); \node[scale=0.4] at (-8.905925, 2.9) {$14$};
\fill (-9.84807, 1.73651) circle (.2); \node[scale=0.4] at (-10.584807, 0.93651) {$15$};
\fill (-10.0, 0.0) circle (.2); \node[scale=0.4] at (-10.8, -0.5) {$16$};

\fill (-9.84808, -1.73645) circle (.2); \node[scale=0.4] at (-10.84808, -1.73645) {$17$};
\fill (-9.39694, -3.42017) circle (.2); \node[scale=0.4] at (-10.39694, -3.42017) {$18$};
\fill (-8.19154, -5.73574) circle (.2); \node[scale=0.4] at (-9.19154, -5.73574) {$19$};
\fill (-7.07109, -7.07104) circle (.2); \node[scale=0.4] at (-8.07109, -7.07104) {$20$};

\fill (-5.00003, -8.66024) circle (.2); \node[scale=0.4] at (-6.00003, -9.16024) {$21$};
\fill (-3.42024, -9.39691) circle (.2); \node[scale=0.4] at (-4.22024, -9.99691) {$22$};
\fill (-0.0000398038, -10.) circle (.2); \node[scale=0.4] at (-0.7, -10.5) {$23$};
\fill (1.73644, -9.84808) circle (.2); \node[scale=0.4] at (2.33644, -10.84808) {$24$};

\fill (4.22614, -9.0631) circle (.2); \node[scale=0.4] at (4.22614, -10.0631) {$25$};
\fill (6.42784, -7.66047) circle (.2); \node[scale=0.4] at (6.42784, -8.96047) {$26$};
\fill (8.66023, -5.00004) circle (.2); \node[scale=0.4] at (8.76023, -6.40004) {$27$};
\fill (9.65925, -2.58824) circle (.2); \node[scale=0.4] at (9.85925, -3.58824) {$28$};

\fill (9.96194, -0.87161) circle (.2); \node[scale=0.4] at (10.56194, -0.57161) {$1$};

\end{scope}

\begin{scope}[xshift=36cm]

\draw[line width=1.0]
   (18.9277, -1.65606) .. controls +(175:1) and +(355:1) .. (15.9391, -1.39458) 
	                     .. controls +(175:1) and +(355:1) ..  (9.96194, -0.87161)
				               .. controls +(175:5) and +(355:6) ..  (0, 0);
\draw[->, line width=1.0]
   (15.9391, -1.39458)  .. controls +(175:0.05) and +(355:0.05) .. (15.9391, -1.39458) ;

\draw[line width=1.0]
   (0,0) .. controls +(10:2) and +(210:3) .. (9.65926, 2.58819)
         .. controls +(30:3) and +(-45:4) ..  (12.2873, 8.60364)
				 .. controls +(135:4) and +(60:3) ..  (5.73577, 8.19152)
				 .. controls +(240:3) and +(60:2) ..  (0.0, 0.0);
\draw[->, line width=1.0]
   (12.2873, 8.60364) .. controls +(-45:0.05) and +(135:0.05) .. (12.2873, 8.60364);

\draw[line width=0.75, color=gray]
   (0,0) .. controls +(20:2) and +(210:2) .. (9.06308, 4.22618)
         .. controls +(30:2) and +(-45:2) ..  (10.649, 7.45649)
				 .. controls +(135:2) and +(60:2) ..  (6.42788, 7.66044)
				 .. controls +(240:2) and +(50:2) ..  (0.0, 0.0);
\draw[->, line width=0.75, color=gray]
   (10.649, 7.45649) .. controls +(-45:0.05) and +(135:0.05) .. (10.649, 7.45649);

\draw[line width=1.0]
   (6.15638, 16.9145) .. controls +(250:1) and +(70:1) .. (4.7883, 13.1557) 
	                     .. controls +(250:1) and +(70:1) ..  (3.42021, 9.39692)
				               .. controls +(255:5) and +(70:6) ..  (0, 0);
\draw[->, line width=1.0]
   (4.7883, 13.1557) .. controls +(250:0.05) and +(70:0.05) .. (4.7883, 13.1557);

\draw[line width=0.75, color=gray]
   (8.15638, 14.9145) .. controls +(270:4) and +(130:5) .. (14.2873, 10.60364)
           .. controls +(-50:5) and +(5:4) ..  (10.0,0)
					 .. controls +(175:4) and +(2:4) ..  (0.0,0);
\draw[->, line width=0.75, color=gray]
    (14.2873, 10.60364) .. controls +(-50:0.05) and +(130:0.05) ..  (14.2873, 10.60364);

\draw[line width=1.0]
   (0,0) .. controls +(90:2) and +(290:3) .. (-0.871543, 9.96195)
         .. controls +(110:3) and +(45:4) ..  (-6.33925, 13.5946)
				 .. controls +(225:4) and +(120:3) ..  (-7.07105, 7.07108)
				 .. controls +(300:3) and +(135:2) ..  (0.0, 0.0);
\draw[->, line width=1.0]
   (-6.33925, 13.5946) .. controls +(225:0.05) and +(45:0.05) .. (-6.33925, 13.5946);

\draw[line width=0.75, color=gray]
   (0,0) .. controls +(95:2) and +(290:2) .. (-2.58818, 9.65926)
         .. controls +(110:2) and +(45:2) ..  (-5.49402, 11.782)
				 .. controls +(225:2) and +(120:2) ..  (-5.73575, 8.19153)
				 .. controls +(300:2) and +(130:2) ..  (0.0, 0.0);
\draw[->, line width=0.75, color=gray]
  (-5.49402, 11.782) .. controls +(225:0.05) and +(45:0.05) .. (-5.49402, 11.782);

\draw[line width=1.0]
   (0,0) .. controls +(155:2) and +(-20:3) .. (-8.66024, 5.00002)
         .. controls +(160:3) and +(70:2) ..  (-14.0954, 5.13034)
				 .. controls +(250:2) and +(160:3) ..  (-9.84807, 1.73651)
				 .. controls +(340:3) and +(175:2) ..  (0.0, 0.0);
\draw[->, line width=1.0]
   (-14.0954, 5.13034) .. controls +(250:0.05) and +(70:0.05) .. (-14.0954, 5.13034);

\draw[line width=0.75, color=gray]
   (0,0) .. controls +(160:2) and +(-20:3) .. (-9.06307, 4.2262)
         .. controls +(160:3) and +(70:1) ..  (-12.216, 4.44629)
				 .. controls +(250:1) and +(160:1) ..  (-9.65925, 2.58821)
				 .. controls +(340:3) and +(170:2) ..  (0.0, 0.0);
\draw[->, line width=0.75, color=gray]
   (-12.216, 4.44629) .. controls +(250:0.05) and +(70:0.05) .. (-12.216, 4.44629);

\draw[line width=0.75, color=gray]
   (12.8906, -10.64187)  .. controls +(225:7) and +(-25:8) .. (-3.60478, -15.7721)
    .. controls +(155:7) and +(350:8) .. (-14.2442, -8.04203)
           .. controls +(170:8) and +(190:4) ..  (-10.0,0.0)
					 .. controls +(10:4) and +(184:4) ..  (0.0,0.0);
\draw[->, line width=0.75, color=gray]
    (-3.60478, -15.7721) .. controls +(155:0.05) and +(-25:0.05) ..  (-3.60478, -15.7721);

\draw[line width=1.0]
   (0,0) .. controls +(185:2) and +(45:3) .. (-9.84808, -1.73645)
         .. controls +(225:3) and +(135:2) ..  (-13.2442, -7.04203)
				 .. controls +(315:2) and +(225:3) ..  (-7.07109, -7.07104)
				 .. controls +(45:3) and +(230:2) ..  (0.0, 0.0);
\draw[->,line width=1.0]
   (-13.2442, -7.04203) .. controls +(135:0.05) and +(315:0.05) .. (-13.2442, -7.04203);

\draw[line width=0.75, color=gray]
   (0,0) .. controls +(190:2) and +(45:1) .. (-9.39694, -3.42017)
         .. controls +(225:1) and +(135:2) ..  (-11.4783, -6.1031)
				 .. controls +(315:2) and +(225:1) ..  (-8.19154, -5.73574)
				 .. controls +(45:1) and +(225:2) ..  (0.0, 0.0);
\draw[->, line width=0.75, color=gray]
   (-11.4783, -6.1031) .. controls +(135:0.05) and +(315:0.05) .. (-11.4783, -6.1031);

\draw[line width=1.0]
   (0,0) .. controls +(245:2) and +(90:3) .. (-5.00003, -8.66024)
         .. controls +(270:3) and +(180:2) ..  (-2.60478, -14.7721)
				 .. controls +(0:2) and +(260:3) ..  (1.73644, -9.84808)
				 .. controls +(80:3) and +(290:2) ..  (0.0, 0.0);
\draw[->,line width=1.0]
   (-2.60478, -14.7721) .. controls +(180:0.05) and +(0:0.05) .. (-2.60478, -14.7721);

\draw[line width=0.75, color=gray]
   (0,0) .. controls +(245:2) and +(90:1) .. (-3.42024, -9.39691)
         .. controls +(270:1) and +(180:2) ..  (-2.25748, -12.8025)
				 .. controls +(0:2) and +(260:1) ..  (-0.0000398038, -10.)
				 .. controls +(80:1) and +(290:2) ..  (0.0, 0.0);
\draw[->, line width=0.75, color=gray]
   (-2.25748, -12.8025) .. controls +(180:0.05) and +(0:0.05) .. (-2.25748, -12.8025);

\draw[line width=1.0]
   (0,0) .. controls +(295:2) and +(130:1) .. (4.22614, -9.0631)
         .. controls +(310:1) and +(225:4) ..  (11.4906, -9.64187)
				 .. controls +(45:4) and +(-30:3) ..  (9.65925, -2.58824)
				 .. controls +(150:3) and +(290:2) ..  (0.0, 0.0);
\draw[->,line width=1.0]
   (11.4906, -9.64187) .. controls +(225:0.05) and +(45:0.05) .. (11.4906, -9.64187);

\draw[line width=0.75, color=gray]
   (0,0) .. controls +(300:2) and +(130:3) .. (6.42784, -7.66047)
         .. controls +(310:3) and +(225:2) ..  (9.95854, -8.35629)
				 .. controls +(45:2) and +(-30:3) ..  (8.66023, -5.00004)
				 .. controls +(150:3) and +(280:2) ..  (0.0, 0.0);
\draw[->, line width=0.75, color=gray]
   (9.95854, -8.35629) .. controls +(225:0.05) and +(45:0.05) .. (9.95854, -8.35629);

\fill[color=gray!110!black] (0,0) circle (.3);

\draw[dashed, line width=0.4] (0,0) circle [radius=10];

\fill (10.0, 0.0) circle (.2); \node[scale=0.4] at (10.5, 0.5) {$2$};
\fill (9.65926, 2.58819) circle (.2); \node[scale=0.4] at (10.2, 2.0) {$3$};
\fill (9.06308, 4.22618) circle (.2); \node[scale=0.4] at (9.7308, 3.8) {$4$};
\fill (6.42788, 7.66044) circle (.2); \node[scale=0.4] at (7.2788, 7.66044) {$5$};
\fill (5.73577, 8.19152) circle (.2); \node[scale=0.4] at (5.53577, 9.19152) {$6$};
\fill (3.42021, 9.39692) circle (.2); \node[scale=0.4] at (4.12021, 9.65692) {$7$};
\fill (-0.871543, 9.96195) circle (.2); \node[scale=0.4] at (-0.841543, 10.96195) {$8$};
\fill (-2.58818, 9.65926) circle (.2); \node[scale=0.4] at (-2.48818, 10.65926) {$9$};
\fill (-5.73575, 8.19153) circle (.2); \node[scale=0.4] at (-5.73575, 9.19153) {$10$};
\fill (-7.07105, 7.07108) circle (.2); \node[scale=0.4] at (-7.07105, 8.07108) {$11$};
\fill (-8.66024, 5.00002) circle (.2); \node[scale=0.4] at (-8.76024, 5.90002) {$12$};
\fill (-9.06307, 4.2262) circle (.2); \node[scale=0.4] at (-9.8307, 3.8262) {$13$};
\fill (-9.65925, 2.58821) circle (.2); \node[scale=0.4] at (-8.905925, 2.9) {$14$};
\fill (-9.84807, 1.73651) circle (.2); \node[scale=0.4] at (-10.584807, 0.93651) {$15$};
\fill (-10.0, 0.0) circle (.2); \node[scale=0.4] at (-10.7, -0.6) {$16$};

\fill (-9.84808, -1.73645) circle (.2); \node[scale=0.4] at (-10.84808, -1.73645) {$17$};
\fill (-9.39694, -3.42017) circle (.2); \node[scale=0.4] at (-10.39694, -3.42017) {$18$};
\fill (-8.19154, -5.73574) circle (.2); \node[scale=0.4] at (-9.19154, -5.73574) {$19$};
\fill (-7.07109, -7.07104) circle (.2); \node[scale=0.4] at (-8.07109, -7.07104) {$20$};

\fill (-5.00003, -8.66024) circle (.2); \node[scale=0.4] at (-6.00003, -9.16024) {$21$};
\fill (-3.42024, -9.39691) circle (.2); \node[scale=0.4] at (-4.22024, -9.99691) {$22$};
\fill (-0.0000398038, -10.) circle (.2); \node[scale=0.4] at (-0.7, -10.5) {$23$};
\fill (1.73644, -9.84808) circle (.2); \node[scale=0.4] at (2.33644, -10.84808) {$24$};

\fill (4.22614, -9.0631) circle (.2); \node[scale=0.4] at (4.22614, -10.0631) {$25$};
\fill (6.42784, -7.66047) circle (.2); \node[scale=0.4] at (6.42784, -8.96047) {$26$};
\fill (8.66023, -5.00004) circle (.2); \node[scale=0.4] at (8.76023, -6.40004) {$27$};
\fill (9.65925, -2.58824) circle (.2); \node[scale=0.4] at (9.85925, -3.58824) {$28$};

\fill (9.96194, -0.87161) circle (.2); \node[scale=0.4] at (10.56194, -0.50) {$1$};

\end{scope}

\end{tikzpicture}

\caption{Two non-equivalent phase portraits with compatible
separatrix skeleton (numbering indicating the compatibility
bijection). \label{compatiblenobasta}}
\end{figure}
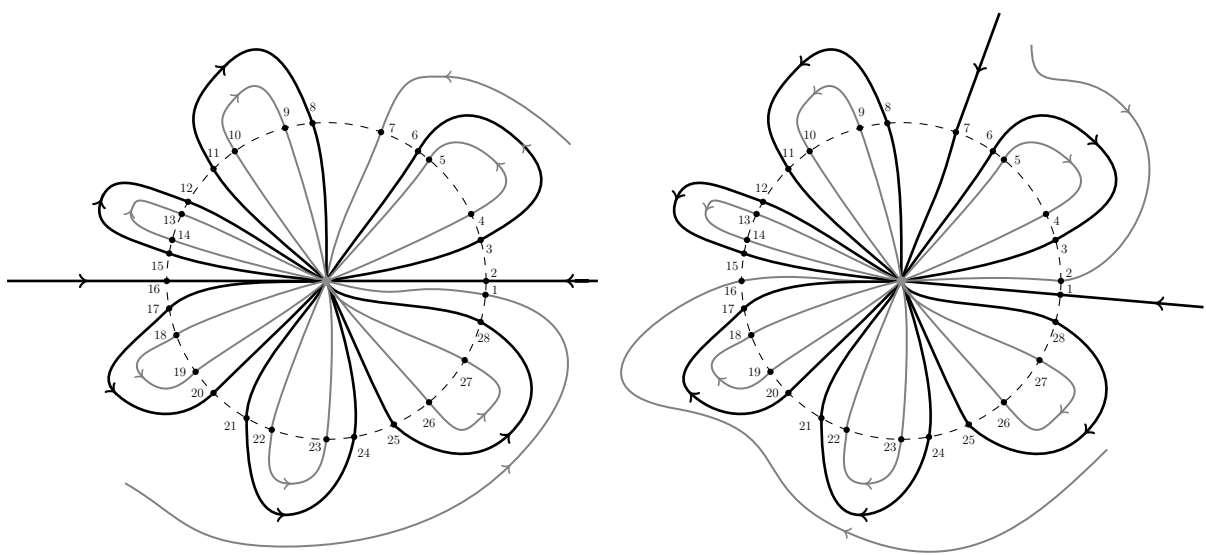

Although the lemmas in Section~\ref{general} do not require
finiteness of separatrices, no attempt has been done to find a
more general version of Theorem~\ref{principal} disposing of this
restriction. Anyway, we are mainly interested in polynomial
(local) flows, that is, those associated to polynomial vector
fields, hence finiteness of separatrices is guaranteed
(Remarks~\ref{casotipico} and \ref{FSDP}). Our next result,
together with Theorem~\ref{principal}, implies that if a flow has
a globally attracting singular point, then it is equivalent to a
polynomial flow.

%
%

\begin{maintheorem}
 \label{ejemploanalitico}
 Let $L$ be a feasible set. Then there are a
 polynomial flow $\Phi$ (having $\0$ as a non-positively stable
 global attractor) and a heteroclinic separatrix $\Sigma$ of $\Phi$
 such that $L$ is the canonical feasible set associated
 to $\Phi$, the counterclockwise orientation in $\RR$ and $\Sigma$.
\end{maintheorem}

Our proof of Theorem~\ref{ejemploanalitico} depends heavily on the
paper \cite{SS}, where sufficient conditions are given allowing
the associated flow to a $C^1$-vector field to be equivalent to a
polynomial flow. It is worth emphasizing that these conditions, as
explained in that paper, are not necessary: fortunately, the
partial result in \cite{SS} turns out to be enough for our
purposes. Still, this is not fully satisfying, because the
arguments in \cite{SS} are essentially non-constructive. In fact,
to the best of our knowledge, the literature provides no explicit
examples of polynomial flows having a non-trivial globally
attracting singular point. For this reason we finally prove:

\begin{maintheorem}
 \label{explicito}
 The origin is both a global attractor and an elliptic saddle
 for the system
 \begin{equation}
  \label{eqex}
  \begin{cases}
 x'=-((1+x^2)y + x^3)^5, \\
 y'=y^2(y^2+x^3).
 \end{cases}
 \end{equation}
\end{maintheorem}

\section{Preliminary notions}
 \label{preliminary}

A number of standard topological notions will be of repeated use
in this paper. We say that a topological space is an \emph{arc}
(respectively, \emph{open arc}, \emph{circle}, \emph{disk}) if it
is homeomorphic to $[0,1]$ (respectively, $\R$, the unit circle
$\mathbb{S}^1=\{(x,y)\in \RR: x^2+y^2=1\}$, the unit disk
$\{(x,y)\in \RR: x^2+y^2\leq 1\}$). If $T$ is arc, and $h:[0,1]\to
T$ is a homeomorphism, then $h(0)$ and $h(1)$ are called the
\emph{endpoints} of $T$. A \emph{region} of a topological space
$X$ is an open, connected subset of $X$.

A \emph{local flow} on a metric space $(X,d)$ is a continuous map
$\Phi:\Lambda\subset \R\times X\rightarrow X$ satisfying:
\begin{itemize}
 \item $\Lambda$ is open in $\R\times X$; moreover, for any $z\in
 X$ the set of numbers $t$ for which $\Phi(t,z)$ is defined is an
 open interval $I_z=(a_z,b_z)$, with
 $-\infty\leq a_z<0<b_z\leq \infty$;
 \item $\Phi(0,z)=z$ for any $z\in X$;
 \item if $\Phi(t,z)=u$, then $I_u=\{s-t:s\in I_z\}$; moreover,
 $\Phi(r,u)=\Phi(r,\Phi(t,z))=\Phi(r+t,z)$ for every $r\in I_u$.
\end{itemize}
In the particular case $\Lambda=\R\times X$, we call $\Phi$ a
\emph{flow} on $X$. Observe that if $X$ is compact, then $I_z=\R$
for any $z\in X$, that is, any local flow on $X$ is a flow. We
write $\Phi_z(t)=\Phi_t(z)=\Phi(t,z)$ whenever it makes sense,
when observe that if $\Phi$ is a flow, then the map
$\Phi_t:X\rightarrow X$ is a homeomorphism for every $t$. We call
$\varphi_{\Phi}(z):=\Phi_{z}(I_z)$. Here (as for the subsequent
notions) we typically omit $\Phi$ in the subindex and write
$\varphi(z)$ instead. If $\varphi(z)=\{z\}$ (when $I_z=\R$), then
we call $z$ a \emph{singular point} of $\Phi$; otherwise the
orbit, and its points, are called \emph{regular}. Since orbits
foliate the space, that is, distinct orbits are disjoint, no point
can be regular and singular at the same time. When the orbit
$\varphi(z)$ is a circle (equivalently, the map $\Phi_z(t)$ is
periodic), it is called \emph{periodic}. If $I\subset I_z$ is an
interval, then we call $\Phi_{z}(I)$ a \emph{semi-orbit} of
$\varphi(z)$. In the particular cases $I=[a,b]$ (with
$\Phi_z(a)=p$, $\Phi_z(b)=q$) $I=[0,b_z)$ or $I=(a_z,0]$, we
rewrite $\Phi_z(I)$ as $\varphi(p,q)$, $\varphi(z,+)$ or
$\varphi(-,z)$,
respectively. We define the
\emph{$\omega$-limit set} of the orbit $\varphi(z)$ (or the point
$z$) as the set
 $$
 \omega(z)=\{u\in X:\exists t_n\to b_z; \Phi_z(t_n)\to u\}.
 $$
The \emph{$\alpha$-limit set} $\alpha(z)$  is analogously defined
(now $t_n\to a_z$).

We say that an orbit $\Gamma$  is \emph{positively} (respectively,
\emph{negatively}) \emph{stable} if for any  $p\in \Gamma$ and any
$\epsilon>0$ there is a number $\delta>0$ (depending of $p$ and
$\epsilon$) such that if $d(p,q)<\delta$, then all points from
$\varphi(q,+)$ (respectively, $\varphi(-,q)$) stay at a distance
less than $\epsilon$ from $\varphi(p,+)$ (respectively,
$\varphi(-,p)$).  We say that $\Gamma$ is \emph{stable} if it is
both positively and negatively stable, and we say that it is
\emph{unstable} if it is not stable. It is worth emphasizing that
these notions are not purely topological: they depend on the
metric $d$.

A set $\Omega\subset X$ is \emph{invariant} for $\Phi$ if it is
the union of some orbits of $\Phi$. If the restriction of $\Phi$
to $\Lambda\cap(\R\times \Omega)$ is a local flow on $\Omega$ (for
instance, if $\Omega$ is invariant), then we call it, more simply
if somewhat incorrectly, the \emph{restriction of $\Phi$ to
$\Omega$}.

Let $\Phi$ and $\Psi$ be respective local flows on the spaces $X$
and $Y$. We say that $\Phi$ and $\Psi$ are \emph{topologically
equivalent} if there is a homeomorphism $h:X \rightarrow Y$ such
that $h(\varphi_\Phi(z))=\varphi_\Psi(h(z))$ for every $z\in X$
which preserves the respective (time) directions  of $\Phi$ and
$\Psi$.

Local flows are associated, in a natural way, to (autonomous)
systems of differential equations defined on smooth manifolds $M$
(which will be seen here as embedded in $\R^m$ for some natural
number $m$). Namely, if $\Phi:\Lambda\subset \R\times M\to M$ is a
smooth local flow, then the vector field $f:M\to \R^m$ given by
$f(z)=\frac{\partial\Phi}{\partial t}(0,z)$ (the \emph{associated
vector field} to $\Phi$) is tangent to $M$ and satisfies
$\frac{\partial\Phi}{\partial t}(t,z)=f(\Phi(t,z))$, that is, the
solution of the system $u'=f(u)$ with initial condition $u(0)=z$
is the map $\Phi_{z}(t):=\Phi(t,z)$, $t\in I_{z}$. Conversely, if
a vector field $f:M\to \R^m$ is tangent to $M$ and sufficiently
smooth (locally Lipschitz is enough), and $\Phi_z(t)$ denotes the
solution of $u'=f(u)$ satisfying $u(0)=z$, then
$\Phi(t,z):=\Phi_z(t)$ is a local flow on $M$. While polynomial
vector fields are the primary interest of this paper, and their
associated flows are usually just local, there is a way to get rid
of this restriction. In fact, if $X$ is locally compact, $O\subset
X$ is open, and $\Phi$ is a local flow on $O$, then there is a
flow $\Psi$ on $X$ whose restriction to $O$ has the same orbits
and  directions as those of $\Phi$, and having singular points
outside $O$ \cite[Lemma~2.3]{JS}. To simplify the notation we will
call $\Phi$, rather than $\Psi$, this extended flow, hoping that
this will not lead to confusion. If $\Phi$ is associated to a
polynomial vector field, then we also call it (and its extension)
\emph{polynomial}, although of course this map is not
``polynomial'' in the usual sense.

In concrete, we are interested in the case $O=\RR$ and
$X=\RR_\infty= \RR\cup\{\infty\}$ (the one-point compactification
of $\RR$), when after identifying $\RR_\infty$ with  the Euclidean
unit sphere $\mathbb{S}^2\subset \R^3$ via the stereographic projection
$(u,v,w)\mapsto(x,y)$ given by $x=u/(1-w)$, $y=v/(1-w)$, we use in
$\RR_\infty$ (and then in $\RR$) the distance $d(\cdot,\cdot)$
inherited from the Euclidean distance in $\mathbb{S}^2$. Hence, the
topologies on $\RR$ and $\RR_\infty$ are the usual ones but
$d(z,z')\leq 2$ for any $z,z'\in \RR_\infty$. Unless otherwise
stated, topological properties of subsets of $\RR$ refer to the
topology in $\RR$. In particular, we mean $A\subset \RR$ to be
\emph{bounded} in the conventional sense, that is, when it is
contained in an Euclidean plane ball (while, of course, all sets
in $\RR$ are ``bounded'' regarding the distance $d(\cdot,\cdot)$).

Sphere and plane local flows have, as it is well known, some
particularly good dynamical properties. The reader is assumed to
be familiar with the basic facts of the Poincar\'e-Bendixson
theory; for instance, recall that if $z$ is regular, then there is
a transversal to $z$ for this flow. (If a local flow $\Phi$ can be
restricted to a neighbourhood $A$ of $z$ so that it is
topologically equivalent to that induced by the constant vector
field $f_0=(1,0)$ on the square $S=(-1,1)\times [-1,1]$, and the
arc $T\subset A$ is the image of the vertical arc $\{0\}\times
[-1,1]$ by the corresponding homeomorphism $h:S\to A$, with
$h(\0)=h(0,0)=z$, then $T$ is called a \emph{transversal} to $z$
for $\Phi$, or just a transversal to $\Phi$ ---or simply a
transversal--- when no emphasis on $z$ is required. If all subarcs
of an open arc or a circle $Q$ are transversal to $\Phi$, we
similarly say that $Q$ is \emph{transversal} to $\Phi$.)

There is a natural way to transport polynomial vector fields from
$\mathbb{S}^2$ to $\RR$. Namely, if $f:\mathbb{S}^2 \to \R^3$ is a polynomial vector
field, tangent to $\mathbb{S}^2$ and vanishing at the north pole $(0,0,1)$
of $\mathbb{S}^2$, say $f(u,v,w)=(P(u,v,w),Q(u,v,w),R(u,v,w))$, then we can
carry it, via the stereographic projection, to the plane vector
field
 $$
 g(x,y)=(1-w)^{-1}(P(u,v,w)+R(u,v,w)x, Q(u,v,w)+R(u,v,w)y)
 $$
with $u=2x/(1+x^2+y^2)$, $v=2y/(1+x^2+y^2)$,
$w=(x^2+y^2-1)/(1+x^2+y^2)$, and after multiplying $g$ by a
appropriate power of $1+x^2+y^2$ we obtain a polynomial vector
field whose associated (polynomial) flow is topologically
equivalent to the flow induced by $f$ on $\mathbb{S}^2\setminus
\{(0,0,1)\}$.

\subsection{On special flows and regions}
 \label{specialflows}

The standing assumption in this paper is that $\0$ is a globally
attracting singular point for the flows $\Phi$ on $\RR$ we deal
with, that is, $\omega(z)=\{\0\}$ for any $z\in \RR$. This is
closely related to the notions of heteroclinicity and
homoclinicity. We say that an orbit $\varphi(z)$ of $\Phi$ is
\emph{homoclinic} (respectively, \emph{heteroclinic}) if (besides
$\omega(z)=\{\0\}$) we have $\alpha(z)=\{\0\}$ (respectively,
$\alpha(z)=\emptyset$
---that is, $\alpha(z)=\{\infty\}$ when
using the extended flow to $\RR_\infty$). Of course, the singular
point $\0$ is trivially homoclinic. If $\Gamma$ is homoclinic,
then we denote by $E(\Gamma)$ the disk enclosed by the circle
$\Gamma\cup \{\0\}$ (or just the singleton $\{\0\}$ in the case
$\Gamma=\{\0\}$). Since $\0$ as a global attractor, any orbit of
$\Phi$ is either heteroclinic or homoclinic
(Lemma~\ref{lounoolootro}).

Let $f_i:\RR\rightarrow \RR$, $1\leq i\leq 4$, be the vector
fields $f_1(x,y)=(x,-y)$, $f_2(x,y)=(-x,-y)$, $f_3(x,y)=(x,y)$,
$f_4(x,y)= (x^2-2xy, xy-y^2)$ respectively. Also, let
$$
A_1=\{(x,y)\in\RR: 0\leq x,y<1, xy<1/2\},
$$
$$
A_2=A_3=A_4=\{(x,y)\in\RR: 0\leq x,y<1, x^2+y^2<1\}.
$$
We remark that although the sets $A_i$ are not open, $f_i$ still
induces a local flow  $\Phi_i$ on $A_i$, $1\leq i\leq 4$. See
Figure~\ref{sectores}. Assume now that $B$ is a set containing
$\0$ and $\Phi$ induces a local flow on $B$ which is topologically
equivalent to $\Phi_i$. Then we say that $B$ is a
\emph{hyperbolic}, \emph{attracting}, \emph{repelling} or
\emph{elliptic sector} of $\Phi$ (at $\0$) when, respectively,
$i=1,2,3,4$. The flow $\Phi$ is said to have the \emph{finite
sectorial decomposition property} (at $\0$) if either $\0$ is
positively stable or has a neighbourhood which is the (minimal)
union of at least two, but finitely many, hyperbolic, attracting,
repelling and elliptic sectors (since $\Phi$ admits no periodic
orbits, see also Proposition~\ref{atractorestable}, this amounts
to the standard definition to be found, for instance, in \cite[p.
18]{DLA}).

\begin{remark}
 \label{casotipico}
The typical case for this to happen is that $\Phi$ is associated
to a vector field (real) analytic at $\0$, see for instance
\cite[Chapter~3]{DLA}.
\end{remark}

\begin{figure}
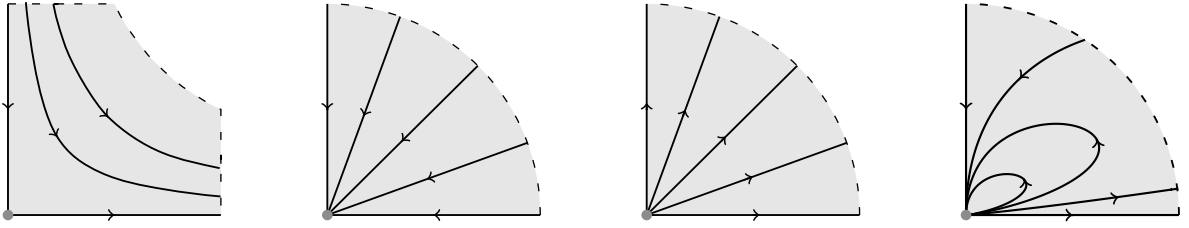

\centering

\caption{From left to right: a hyperbolic,
an attracting, a repelling and
an elliptic sector.} \label{sectores}
\end{figure}

We call a region $\Omega\subset \RR$  \emph{radial} (respectively,
a \emph{strip}) if it is invariant for $\Phi$, and, when
restricted to $\Omega$, $\Phi$ is topologically equivalent to the
flow induced by $f_2$ on $\RR\setminus \{\0\}$ (respectively, in
the upper half-plane $\mathbb{H}= \R\times (0,\infty)$). Needless
to say, to define strips, one can equivalently use (as it is
usually done) the associated flow to the constant vector field
$f_0$ on $\RR$. If all orbits of a strip $\Omega$ are heteroclinic
(respectively, homoclinic), then we call $\Omega$
\emph{heteroclinic} (respectively, \emph{homoclinic}) as well.
Observe that, in general, the interior of a hyperbolic, attracting
or repelling sector is not a strip because it is not invariant (it
does not consist of full orbits of $\Phi$). We say that the strip
$\Omega$ is \emph{strong} if there are orbits $\Gamma_1,\Gamma_2$
in $\Bd\Omega$ such that the restriction of $\Phi$ to $\Omega\cup
\Gamma_1\cup \Gamma_2$ is equivalent to that of the flow induced
by $f_2$ on $\Cl \mathbb{H}\setminus \{\0\}$. If, moreover,
$\Cl\Omega=\Omega\cup \Gamma_1\cup\Gamma_2\cup \{\0\}$, then we
say that $\Omega$ is \emph{solid}.

 \begin{remark}
If $\Omega$ is a solid strip, then either all $\Omega$, $\Gamma_1$
and $\Gamma_2$ are heteroclinic, or all of them are homoclinic.
Otherwise, as it is easy to check, either (a) one of orbits, say
$\Gamma_1$, is heteroclinic, $\Gamma_2$ is homoclinic and
$\Omega=\R^2\setminus (\Gamma_1\cup E(\Gamma_2))$, or (b) both
$\Gamma_1$ and $\Gamma_2$ are homoclinic, with $E(\Gamma_1)\cap
E(\Gamma_2)=\{\0\}$, and $\Omega=\R^2\setminus (E(\Gamma_1)\cup
E(\Gamma_2))$. Use Lemma~\ref{bounded} to find a heteroclinic
orbit $\Gamma\subset \Omega$. Clearly, $\Gamma$ cannot disconnect
$\Omega\cup \Gamma_1\cup \Gamma_2$, which contradicts that
$\Omega$ is strong.
 \end{remark}

If $Q$ is a transversal circle (respectively, open arc) with the
property that, for every $z\in Q$, $\varphi(z)$ intersects $Q$
exactly at $z$, then $\Omega=\bigcup_{z\in Q} \varphi(z)$ is
radial (respectively, a strip). To construct the corresponding
homeomorphism $h:\RR\setminus \{\0\}\to \Omega$ (respectively,
$h:\mathbb{H}\to \Omega$) just fix a homeomorphism
$f:\mathbb{S}^1\to Q$ (respectively, $f:\mathbb{S}^1\cap
\mathbb{H}\to Q$) and write
$h(e^{-t+\ii\theta})=\Phi(t,f(e^{\ii\theta}))$. Conversely, if
$\Omega\subset \RR$ is radial (respectively, a strip)  then there
is a circle (respectively, an open arc) $Q\subset \Omega$,
transversal to $\Phi$, having exactly one common point with every
orbit in $\Omega$. We call any such set $Q$ a \emph{complete
transversal} to $\Omega$.  If $\Omega$ is a strong strip, then
more is true: there is a transversal arc $T$ having exactly one
common point with every orbit in $\Omega$ and every orbit
$\Gamma_1,\Gamma_2$. We call $T$ a \emph{strong transversal} to
$\Omega$.

\begin{remark}
 \label{fundamental}
If $\Omega$ is radial, and the circle $C$ is a complete
transversal to $\Omega$, then it must enclose $\0$. Hence all
heteroclinic orbits intersect $C$, that is, $\Omega$ is the union
set of all heteroclinic orbits of $\Phi$; in other words, $\Phi$
admits one radial region at most (later we will see,
Proposition~\ref{atractorestable}, that such a region does exist).
Moreover, the circles $\Phi_t(C)$ tend uniformly to $\infty$ as
$t\to -\infty$. In fact, if $D_t\subset \RR_\infty$ is the disk
containing $\infty$ and having $\Phi_t(C)$ as its boundary, then
$D_t=\{\Phi_s(u):u\in C, s\leq t\}\cup\{\infty\}$. Since these
disks intersect exactly at $\infty$, we get $\diam(D_t)\to 0$ as
$t\to -\infty$, and the uniform convergence to $\infty$ follows.
As a corollary, all heteroclinic orbits are negatively stable.

Similarly, if $\Omega$ is a solid strip and  $T$ is a strong
transversal to $\Omega$, then $\Phi_t(T)$ tends uniformly to $\0$
as $t\to \infty$, and tend uniformly to $\0$ as $t\to -\infty$ in
the homoclinic case, and to $\infty$ in the heteroclinic case. In
particular, all orbits of a  solid strip are stable, and if it is
heteroclinic (respectively, homoclinic), then the flow induced by
$f_2$ on $\Cl \mathbb{H}$ (respectively, by $f_4$ on the union set
of $\0$ and all orbits intersecting the diagonal arc
$\{(x,x):1/2\leq x\leq 1\}$) is topologically equivalent to the
restriction of $\Phi$ to $\Cl\Omega$.
\end{remark}

If an orbit is not contained in any solid strip, then it is called
a \emph{separatrix} of $\Phi$. Note that the union set $X$ of all
separatrices of $\Phi$ is closed. The components of $\RR\setminus
X$ are called the \emph{canonical regions} of $\Phi$. A family of
orbits of $\Phi$ consisting of all its separatrices and exactly
one orbit from every canonical region is called a \emph{separatrix
skeleton} of $\Phi$. Observe that any regular separatrix can
belong to the boundary of, at most, two different canonical
regions. Therefore, if the number of separatrices is finite, so it
the number of canonical regions.

\begin{remark}
 \label{FSDP}
As indicated in Remark~\ref{fundamental}, any unstable orbit must
be a separatrix. If $\Phi$ has the finite sectorial decomposition
property, then $\Gamma$ is a separatrix if and only if it is
either the singular point, or includes a semi-orbit limiting a
hyperbolic sector. In particular, $\Phi$ has finitely many
separatrices and $\Gamma$ is a separatrix if and only if it is
unstable.
\end{remark}

The next result is a particular case of \cite[Theorems~5.2 and
7.1]{Ma}, see also \cite{Ne} and \cite{EJ}:

\begin{proposition}
 \label{canonical}
 Any canonical region of $\Phi$ is either radial or a strip.
\end{proposition}

\begin{remark}
 \label{h-h}
A strip (even a strong strip) needs not be either heteroclinic or
homoclinic. Nevertheless, if a canonical region is a strip, then
it must be either heteroclinic or homoclinic (because, in this
case, the set of its heteroclinic orbits and the set of its
homoclinic orbits are both open; hence, by connectedness, one of
them must be empty).
\end{remark}

\begin{theorem}
 \label{markus}
Assume that $\0$ is a global attractor for two flows $\Phi$ and
$\Phi'$ and let  $\mathcal{X}$ and $\mathcal{X'}$ denote some
separatrix skeletons for $\Phi$ and $\Phi'$. Then $\Phi$ and
$\Phi'$ are topologically equivalent if and only if there is a
homeomorphism from the plane onto itself mapping the orbits of
$\mathcal{X}$ onto the orbits of $\mathcal{X'}$ and preserving the
flows directions.
\end{theorem}

\begin{remark}
 \label{problema}
Our definition of separatrix is not the standard one (compare to
\cite{Ma}, \cite{Ne}, \cite[p. 294]{Pe} or \cite[p. 34]{DLA}),
even when we restrict ourselves, as it is the case here
(Lemma~\ref{lounoolootro}), to flows having only heteroclinic or
homoclinic orbits. More precisely, our ``separatrices'' are what
we called ``separators'' in \cite{EJ} (and our  ``canonical
regions'' what we called ``standard regions'' there).  If the
boundary of a heteroclinic strip consists of the singular point
and two heteroclinic orbits, then it is solid (the corresponding
strong transversal can be found with the help of
Lemma~\ref{transversal}). If we replace ``heteroclinic'' by
``homoclinic'', this needs not happen unless we additionally
assume that the ordering ``$\prec$'' we introduce below
Lemma~\ref{bounded} totally orders the orbits of the closure of
the strip. This point is missed in the above-mentioned references
and, as a consequence, Theorem~\ref{markus}, as stated there, does
not work, see \cite{EJ} for the details. Surprisingly, it seems
that this fact has passed unnoticed until now.
\end{remark}

\subsection{On orientations and the extension of homeomorphisms}
 \label{extensiones}

Let $C$ be a circle around $\0$. If $\Gamma$ is heteroclinic, we
call the last point of $\Gamma$ in $C$ (that is, the point $q\in
\Gamma\cap C$ such that $\Phi_q(t)\notin C$ for any $t>0$) the
\emph{$\omega$-point} of $\Gamma$ in $C$. Likewise, if $\Gamma$ is
regular and homoclinic and $C$ is small enough so that there are
points of $\Gamma$ not enclosed by $C$, then we call the first and
 last points of $\Gamma$ in $C$ (that is, the points $p,q\in
\Gamma\cap C$ such that $\Phi_p(t)\notin C$ for any $t<0$ and
$\Phi_q(t)\notin C$ for any $t>0$) the \emph{$\alpha$-point} and
the \emph{$\omega$-point} of $\Gamma$ in $C$, respectively.

If $\mathcal{P}$ is a finite family of orbits of $\Phi$, and $C$
is a circle around $\0$ small enough, then we denote by
$\Delta_\Phi(\mathcal{P},C)$ the set of all $\alpha$- and
$\omega$-points in $C$ from the orbits in $\mathcal{P}$ and call
it the \emph{configuration} of $\mathcal{P}$ in $C$. Note that the
possibility that the singular point belongs to $\mathcal{P}$ is
not excluded, when of course it adds no points to
$\Delta_\Phi(\mathcal{P},C)$. Also, observe that all
configurations of $\mathcal{P}$ are essentially the same, that is,
if $C$ and $C'$ are small circles around $\0$, then there is an
orientation preserving homeomorphism $h:C\to C'$ mapping the
$\alpha$- and $\omega$-points in $C$ of every orbit $\Gamma\in
\mathcal{P}$ to the $\alpha$- and $\omega$-points in $C'$ of that
same orbit $\Gamma$.

We call  a triplet $(A,B,C)$ of arcs in $\RR_\infty$ sharing a
common endpoint $p$ (and no other point) a \emph{triod}. The point
$p$ is called the \emph{vertex} of the triod, the other endpoints
of the arcs $A,B,C$ being called its \emph{endpoints.} We say that
the triod $(A,B,C)$  is \emph{positive}, when, after taking an
open euclidean ball $U$ of center $p$ and radius $\epsilon>0$
small enough, there is $\theta_0\in \mathbb{R}$ such that the
first intersection points of these arcs with $\Bd U$ can be
written as $p+\epsilon e^{\ii\theta_A}, p+\epsilon
e^{\ii\theta_B}, p+\epsilon e^{\ii\theta_C}$, with
$\theta_0=\theta_A<\theta_B<\theta_C<\theta_0+2\pi$.  We say that
the triod is \emph{negative} when it is not positive. Observe that
the definition above excludes the case when the common endpoint
$p$ is $\infty$. We then say that $(A,B,C)$ is positive when
$(G(A),G(B),G(C))$ is negative, $G:\RR_\infty\to \RR_\infty$ being
defined by $G(z)=1/\overline{z}$ (here we identify $\RR$ with
$\mathbb{C}$ and mean $G(\infty)=\0$, $G(\0)=\infty$). If $C$ is a
circle around $\0$ and $(q,q',q'')$ is a triplet of distinct
points in $C$, then we call it \emph{positive} or \emph{negative}
according to whether it is counterclockwise or clockwise oriented
in $C$, that is, there is a positive (negative) triod $(A,A',A'')$
in the disk enclosed by $C$ with vertex $\0$ and endpoints
$q,q',q''$. If $\Gamma$ is homoclinic, then we say that it is
\emph{positive} (respectively, \emph{negative}) when, after taking
$\Gamma'\subset \Inte E(\Gamma)$ and a small circle $C$ around
$\0$, the $\alpha$- and $\omega$-points $p,q$ of $\Gamma$ in $C$,
and the $\omega$-point $q'$ of $\Gamma'$ in $C$, we get that
$(p,q',q)$ is positive (respectively, negative). In simpler words,
$\Gamma$ is positive (negative) when the flow induces the
counterclockwise (clockwise) orientation on $\Gamma\cup \{\0\}$.

Let $P,P'\subset \RR$ (respectively, $P,P'\subset\RR_\infty$). We
say that $P$ and $P'$ are \emph{$\RR$-compatible} (respectively,
\emph{$\RR_\infty$-compatible}) if there is a homeomorphism $H$
from $\RR$ (respectively, $\RR_\infty$) onto itself  mapping $P$
onto $P'$. Clearly, $\RR$-homeomorphisms amount to
$\RR_\infty$-homeomorphisms mapping $\infty$ to itself. If
$H:\RR_\infty\to \RR_\infty$ is a homeomorphism, then, as it is
well known, either it preserves the orientation, that is, all
pairs of triods $(A,B,C)$ and $(H(A),H(B),H(C))$ have the same
sign, or it reverses the orientation, that is, all pairs of triods
$(A,B,C)$ and $(H(A),H(B),H(C))$ have opposite sign. As it turns
out, see \cite{AM}, this is the key property to identify
compatibility: two Peano sets $P$ and $P'$ (by a \emph{Peano
space} we mean a compact, connected, locally connected set) in
$\RR_\infty$ are $\RR_\infty$-compatible if and only if there is a
homeomorphism $h:P\to P'$ either preserving or reversing the
orientation, in the former sense, for all pair of triods $(A,B,C)$
and $(h(A),h(B),h(C))$ in $P$ and $P'$ (when $h$ can indeed be
homeomorphically extended to the whole $\RR_\infty$).

The former result can be adapted to the $\RR$-setting as follows.
We say that $P\subset \RR$ is \emph{nice} if it is unbounded,
$P_\infty=P\cup \{\infty\}$ is a Peano subset of $\RR_\infty$, and
for any triod $(A,B,C)$ in $P_\infty$ with vertex $\infty$ there
is a $\theta$-curve in $P_\infty$ including $A$, $B$ and $C$ (by a
\emph{$\theta$-curve} we mean a union of three arcs intersecting
exactly at their endpoints). Then we get: two nice sets $P,P'$ are
$\RR$-compatible if and only if there is a homeomorphism $h:P\to
P'$ either preserving or reversing the orientation for all pair of
triods $(A,B,C)$ and $(h(A),h(B),h(C))$ in $P$ and $P'$ (when,
again, $h$ can indeed be homeomorphically extended to the whole
$\RR$).

Assume that $\mathcal{P}$ and $\mathcal{P}'$ are finite families
of orbits of, respectively, $\Phi$ and $\Phi'$ (we also assume
that both of them contain the globally attracting singular point
$\0$ and at least one heteroclinic and one homoclinic orbit). Let
$P$ and $P'$ be the union sets of these orbits and note that these
sets are nice. Then, as it is simple to check, a condition
characterizing the $\RR$-compatibility of $P$ and $P'$ (when we
accordingly say that $\mathcal{P}$ and $\mathcal{P}'$ are
\emph{compatible}) is the existence of a \emph{compatibility
bijection}. By this we mean a bijection $\xi:\mathcal{P}\to
\mathcal{P}'$ for which there is a homeomorphism $\mu:C\to C'$,
with $C$ and $C'$ small circles around $\0$, mapping
$\Delta_\Phi(\mathcal{P},C)$ onto
$\Delta_{\Phi'}(\mathcal{P'},C')$, so that $\mu(C\cap
\Gamma)=C'\cap\xi(\Gamma)$ for any $\Gamma\in \mathcal{P}$. In
this case we say that $\mu$ \emph{preserves orbits for $\xi$}.

If, additionally, $\mu$ maps $\omega$-points onto $\omega$-points
(when we say that $\mu$ \emph{preserves directions for $\xi$}),
then the corresponding plane homeomorphism preserves the flows
directions on $\mathcal{P}$ and $\mathcal{P}'$. If, moreover,
these families are the separatrix skeletons of $\Phi$ and $\Phi'$,
Theorem~\ref{markus} implies that the flows are equivalent.

\subsection{A lemma on Janiszewski spaces}

A compact connected Hausdorff space is called a \emph{continuum}.
We say that a topological space $X$ is a \emph{Janiszewski space}
it is a locally connected continuum and, moreover, for any
subcontinua $C_1,C_2\subset X$ with the property that $C_1\cap
C_2$ is not connected, there are points $x,y\in X\setminus
(C_1\cup C_2)$ which are simultaneously contained in no
subcontinuum in $X\setminus (C_1\cup C_2)$. By \cite[Fundamental
Theorem~6, p. 531]{Ku}, a topological space $X$ is homeomorphic
to $\RR_\infty$ if and only
if it is a Janiszewski space, contains more than one point, and,
for any $x\in X$, the set $X\setminus \{x\}$ is connected. If $X$
is a Janiszewski space, $Y$ is Hausdorff and there is a continuous
monotone map mapping $X$ onto $Y$, then $Y$ is Janiszewski as well
(we say that $f:X\to Y$ is \emph{monotone} if $f^{-1}(A)$ is
connected whenever $A\subset Y$ is connected). In fact, this is
proved in \cite[Theorem~9, p. 507]{Ku} additionally assuming that
$Y$ is a locally connected continuum; but if $X$ is a locally
connected continuum, $Y$ is Hausdorff, and $X$ can be continuously
mapped onto $Y$, then $Y$ is indeed a locally connected continuum,
as seen in \cite[Theorem~9, p. 259]{Ku}.

Let $K\subset \RR_\infty$ be a continuum such that
$\RR_\infty\setminus K$ is connected. We define the equivalence
relation ``$\sim_K$'' in $\RR_\infty$ by $x\sim_K y$ if either
$x=y$ or both $x$ and $y$ belong to $K$. Then we have:

\begin{lemma}
 \label{cociente}
 The quotient space $\mathcal{Q}=\RR_\infty/\sim_K$ is homeomorphic to
 $\RR_\infty$.
\end{lemma}

\begin{proof}
 According to the previous discussion, if
 $\Pi:\RR_\infty\to \mathcal{Q}$ is the projection map (when recall that
 $\mathcal{U}$ is open in $\mathcal{Q}$ if and only if
 $\Pi^{-1}(\mathcal{U})$ is open in $\RR_\infty$), then, in order
 to prove that $\mathcal{Q}$ is homeomorphic to $\RR_\infty$, we just
 need to show:
 \begin{itemize}
  \item[(i)] $\mathcal{Q}$ is Hausdorff;
  \item[(ii)] $\mathcal{Q}\setminus \{X\}$ is connected for any $X\in \mathcal{Q}$.
  \item[(iii)] $\Pi^{-1}(\mathcal{C})$ is connected for any
  connected set  $\mathcal{C}\subset \mathcal{Q}$.
 \end{itemize}
 Statements (i) and (ii) are immediate because of the assumptions on $K$. To prove
 (iii) we use that $\Pi$ is a closed map by (i) and then apply
 \cite[Theorem~9, p. 131]{Ku} and the fact that any $X\in \mathcal{Q}$
 is a connected subset of $\RR_\infty$.
\end{proof}

\section{General results on global attraction}
 \label{general}

Recall that we assume that $\0$ is a global attractor for $\Phi$.

\begin{lemma}
 \label{lounoolootro}
 All orbits of $\Phi$ are either homoclinic or heteroclinic.
\end{lemma}

\begin{proof}
 If the statement of the lemma is not true, then
there is some
 point $z \in \R^2$ such that $\alpha(z)$
contains a regular point $u$. Let $T$ be a transversal to $u$.
According to some well-known Poincar\'{e}-Bendixson
 theory, we can find $p,q\in \varphi(z)\cap T$ so that
 $\varphi(p,q)\cup S$ (where $S$ is the arc in $T$ whose endpoints are
 $p$ and $q$) is a circle enclosing a disk $D$ in $\RR_\infty$ which
 contains $\varphi(-,p)$, and hence $\alpha(z)$, and intersects
 $\varphi(q,+)$ just at $q$. This is impossible: on the one hand,
 $\0$ cannot belong to $D$, because it is the $\omega$-limit set
 of $\varphi(q)$; on the other hand, $u\in\alpha(z)$
 implies $\omega(u)\subset \alpha(z)$, so $\0$ does belong to
 $D$.
\end{proof}

\begin{lemma}
 \label{bounded}
 The union set of
 all homoclinic orbits of $\Phi$ is bounded.
\end{lemma}

\begin{proof}
 Assume the opposite to find a family of homoclinic orbits
 $\{\varphi(z_n)\}_{n=1}^\infty$ with $z_n\to \infty$ as $n\to
 \infty$ and fix a circle $C$ around $\0$. Using the continuity of the
(extended) flow $\Phi$ at $\infty$,
 there is no loss of generality in assuming that the semi-orbits
 $\Phi_{z_n}([-n,0])$ do not intersect the region $O$ encircled by $C$.
 Next, find  the numbers $a_n\leq -n$, closest to $-n$, such that the points
 $\Phi_{z_n}(a_n)$ belong to $C$
 (using that the orbits $\varphi(z_n)$ are homoclinic)
 and assume, again without loss of generality,
 that the points $u_n=\Phi_{z_n}(a_n)$ converge to $u$. Since
 $\Phi_{u_n}(t)\in \RR\setminus O$ for any $t\in [0,n]$, the continuity of
 the flow implies that $\varphi(u,+)$ does not intersect $O$,
 contradicting that $\0$ is a global attractor.
\end{proof}

Let $\mathcal{H}$ denote the family of homoclinic orbits of
$\Phi$. We introduce a partial order in $\mathcal{H}$ by writing
$\Gamma\preceq \Sigma$ if $\Gamma\subset E(\Sigma)$, when
$\Gamma\prec \Sigma$ means of course $\Gamma\preceq \Sigma$ with
$\Gamma\neq \Sigma$. We say that $\Gamma \in \mathcal{H}$ is
\emph{maximal} if there is no $\Sigma\in \mathcal{H}$ such that
$\Gamma\prec \Sigma$. If $\Gamma,\Sigma\in \mathcal{H}$  and
neither $\Gamma\preceq \Sigma$ nor $\Sigma\preceq \Gamma$ is true,
then we say that $\Gamma$ and $\Sigma$ are \emph{incomparable}.
Realize that a family of pairwise incomparable orbits must be
countable. Moreover, we have:

\begin{lemma}
 \label{haciacero}
 If the orbits $\{\Gamma_n\}_{n=1}^\infty$ are pairwise
 incomparable, then $\diam(\Gamma_n)\to 0$ as $n\to \infty$.
\end{lemma}

\begin{proof}
 Suppose the contrary to get a point $u\neq \0$ at which these orbits
 accumulate. Let $T$ be a transversal to $u$  and find points
 $u_{n_k}\in \Gamma_{n_k}\cap T$, $k=1,2,3$, with, say, $u_{n_2}$
 lying between  $u_{n_1}$ and $u_{n_3}$ in $T$. Then $u_{n_1}$ and
 $u_{n_3}$ belong to different regions in $\RR\setminus
 (\Gamma_{n_2}\cup \{\0\})$: we are using here that any homoclinic
 orbit can intersect a transversal at one point at most. Thus,
 either $\Gamma_{n_1}\prec \Gamma_{n_2}$ or $\Gamma_{n_3}\prec
 \Gamma_{n_2}$, contradicting the hypothesis.
\end{proof}

\begin{lemma}
 \label{region}
 Let $\Omega\subsetneqq \RR$ be a region invariant for
 $\Phi$.
 \begin{itemize}
 \item[(i)] If $\Omega$ is bounded, then $\Bd \Omega$ is the union
  set of a homoclinic orbit $\Sigma$, a
  (possibly empty) family $\mathcal{G}$ of pairwise incomparable
  homoclinic orbits  satisfying  $\Gamma\prec \Sigma$ for
  every $\Gamma\in \mathcal{G}$, and the singular point.
 \item[(ii)] If $\Omega$ is unbounded, then its boundary is the
  union set of at most two heteroclinic orbits, a (possibly empty)
    family of pairwise incomparable homoclinic orbits, and the
    singular point.
 \end{itemize}
\end{lemma}

\begin{proof}
 Since $\Omega$ in invariant, $\Bd \Omega$ is invariant as well,
 and the statement (ii) follows easily from the connectedness of $\Omega$.
 To prove (i), assume that the
 boundary of the bounded region $\Omega$ is not as described and
 realize that then we must have
 $\Bd \Omega=
 \{\0\}\cup\bigcup_n \Gamma_n$ for a family
 $\{\Gamma_n\}_n$ (having at least two elements)
 of pairwise incomparable homoclinic orbits.
 Lemma~\ref{haciacero}  implies that $O=\RR\setminus
 \bigcup_n E(\Gamma_n)$ is a region including $\Omega$  with
 the same boundary as $\Omega$. Hence $\Omega=O$, contradicting
 that $\Omega$ is bounded.
\end{proof}

\begin{lemma}
 \label{maximal}
 Let $\Gamma\in \mathcal{H}$. Then there is $\Sigma\in \mathcal{H}$,
 maximal for  ``$\prec$'',
 such that $\Gamma\preceq \Sigma$.
\end{lemma}

\begin{proof}
 If $\Gamma$ is not maximal itself, then the Jordan curve theorem
 implies that the non-empty family $\mathcal{F}=
 \{\Gamma' \in \mathcal{H} : \Gamma \preceq \Gamma'\}$ is a totally ordered
 subset of $\mathcal{H}$; accordingly,
 it is enough to show that $\mathcal{F}$ has a maximal
 element for $\preceq$.

 Say $\mathcal{F}=\{\Gamma_i\}_i$. Then, because
 of the total ordering, $\Omega=\bigcup_i \Inte
 E(\Gamma_i)$ is a region invariant for $\Phi$, and because
 of Lemma~\ref{bounded}, $\Omega$ is bounded. As a result,
 we can apply Lemma~\ref{region}(i) to obtain the corresponding
 homoclinic boundary orbit $\Sigma$. Then, clearly, $\Sigma$
 is the maximal element of $\mathcal{F}$.
\end{proof}

\begin{remark}
 \label{maxhomo}
 Note that all maximal homoclinic orbits of $\Phi$ are separatrices.
\end{remark}

\begin{lemma}
 \label{transversal}
 Let $z$ be a regular point. Then there is a transversal $T$ to
 $z$ such that, for every  $u\in T$, $\varphi(u)$ intersects $T$
 exactly at $u$.
\end{lemma}

\begin{proof}
 Fix an arc $Q$ transversal to $z$. Note that no orbit can intersect $Q$
 infinitely many times. Also, if some orbit
 intersects $Q$ at consecutive times $t<s$ and corresponding
 points $u$ and $v$, then no orbit can intersect the open arc in
 $Q$ with endpoints $u$ and $v$ more than once. Using these two
 facts it is easy to construct a transversal $T\subset Q$ to $z$
 with endpoints $p$ and $q$ such that the orbits $\varphi(p)$ and
 $\varphi(q)$ intersect $T$ at exactly $p$ and $q$. This is the
 transversal we are looking for, because if an orbit $\Gamma$
 consecutively intersects $T$ at points $u$ and $v$, and $D$ is the
 disk in $\RR_\infty$ enclosed by $\varphi(u,v)$ and the arc in $T$ with endpoints
 $u$ and $v$ such that $\0\in D$, then either $\varphi(p)$ or
 $\varphi(q)$ does not intersect $D$, a contradiction.
\end{proof}

\begin{lemma}
 \label{controlado}
 If $\Omega$ is a canonical region and $\Gamma,\Gamma'$ are
 distinct orbits in $\Omega$, then there is
 a solid strip $S\subset \Omega$ such that
 $\Bd S=\Gamma\cup\Gamma'\cup \{\0\}$.
\end{lemma}

\begin{proof}
 Let $Q$ be a complete transversal to $\Omega$ and let $A\subset Q$ be
 an arc with endpoints belonging to $\Gamma$ and $\Gamma'$.
 Since $\Omega$ includes no
 separatrices, for any point $z\in Q$ there is a solid
 strip in $\Omega$, containing $z$, whose closure intersects $Q$
 at a small arc in $Q$ (this small arc thus being a strong
 transversal to the strip). Taking this into account, and applying a
 simple compactness argument to $A$, the lemma follows.
\end{proof}

Recall that $\Phi$ admits one radial region at most, that
consisting of all heteroclinic orbits of $\Phi$
(Remark~\ref{fundamental}). Indeed, such is the case:

\begin{proposition}
 \label{atractorestable}
 Let $R$ be the union set of
 all heteroclinic orbits of $\Phi$.
 Then it is radial. Moreover:
 \begin{itemize}
  \item[(i)] If $R=\RR\setminus
   \{\0\}$, that is, all regular orbits of $\Phi$ are
   heteroclinic, then $\Phi$ is topologically equivalent to the
   associated flow to  $f_2(x,y)=(-x,-y)$ in $\RR$ (hence $\0$ is
   positively stable and it is the only separatrix of $\Phi$).
  \item[(ii)] If $R\neq \RR\setminus \{\0\}$, then $R$ includes
   a separatrix of $\Phi$.
  \end{itemize}
\end{proposition}

\begin{proof}
First we assume  $R=\RR\setminus \{\0\}$. To prove that $R$ is
radial and (i) holds, it suffices to show that $\0$ is the only
 separatrix of $\Phi$
 (Proposition~\ref{canonical} and Theorem~\ref{markus}).
 Take $z\in R$ and let $T\subset R$
 be an arc transversal to $z$ with the property that the orbits of all its
 points intersect $T$ exactly once (Lemma~\ref{transversal}). Let
 $p$ and $q$ be the endpoints of $T$ and let $D$ be the disk in
 $\RR_\infty$ enclosed by $\varphi(p)$, $\varphi(q)$, $\0$ and
 $\infty$ and including $T$. If $u\in \Inte D$, then $\varphi(u)$
 intersect $T$ (because it is heteroclinic). Therefore, $\Inte D$ is a
 heteroclinic solid strip; in particular, $\varphi(z)$ is not a
 separatrix.

 Assume now $R\neq \RR\setminus \{\0\}$.
 Applying Lemma~\ref{cociente} to the union set $K=\RR\setminus R$ of all sets
 $E(\Gamma)$ with $\Gamma$ maximal for ``$\prec$" (recall also
 Lemmas~\ref{haciacero} y \ref{maximal}), and using (i), we can
 construct a topological equivalence between the restriction of
 $\Phi$ to $R$ and the restriction (to  $\RR\setminus \{\0\}$) of the
 associated flow to $f_2$. In particular, $R$ is radial.

 To prove the last statement of the proposition, assume that $R$
 includes no separatrices (hence it is a canonical region by
 Proposition~\ref{canonical}),
 fix a complete transversal circle $C$ to $R$ and apply
 Lemma~\ref{controlado} (recall also Remark~\ref{fundamental})
 to conclude the uniform
 convergence of $\Phi_t(C)$ to $\0$ and $\infty$ as $t\to
 \pm\infty$. Then $R=\bigcup_{t\in\R} \Phi_t(C)=\RR\setminus
 \{\0\}$, a contradiction.
\end{proof}

\section{Proof of Theorem~\ref{principal}}

In this section we assume, besides global attraction, that $\0$ is
not positively stable and $\Phi$ has finitely many separatrices.

Let $\mathcal{X}$  be a separatrix skeleton for $\Phi$, fix a
small circle $C$ around $\0$ and let
$X=\Delta_\Phi(\mathcal{X},C)$ be the configuration of $\Phi$ in
$C$. Also, fix an orientation $\Theta$ (counterclockwise or
clockwise) in $\RR$ and a heteroclinic separatrix $\Sigma$ in
$\mathcal{X}$ (such an orbit exists because of
Proposition~\ref{atractorestable}(ii)). Let $q$ be the
$\omega$-point of $\Sigma$ in $C$. Find disjoint open arcs
$J,J'\subset C$ whose closures have $q$ as their common endpoint
(small enough so that they do not contain any points from $X$),
take points $p\in J$, $p'\in J'$, and assume that they are
labelled so that the orientation of $(p,q,p')$ in $C$ is that
given by $\Theta$ (that is, $(p,q,p')$ is positive  if and only if
$\Theta$ is the counterclockwise orientation). Finally, after
removing $J'$ from $C$, we get an arc $A$ with endpoints $a$ (the
other endpoint of the closure of $J'$) and $q$, and order the
points from $A$ in the natural way so that $a<q$.

We call positive (negative) homoclinic orbits \emph{even} when
$\Theta$ is the counterclockwise (clockwise) orientation, and
\emph{odd} when $\Theta$ is the clockwise (counterclockwise)
orientation. Thus, a homoclinic orbit from $\mathcal{X}$ is even
if and only if its $\alpha$-point $v$ and its $\omega$-point $w$
satisfy $v<w.$ By convention, all heteroclinic orbits are even. We
say that two orbits have the \emph{same parity} when both are even
or both are odd.

According to Proposition~\ref{canonical} and, again,
Proposition~\ref{atractorestable}(ii), all canonical regions are
indeed strips, so we will call them \emph{canonical strips}.
Recall (Remark~\ref{h-h}) that any canonical strip must be either
heteroclinic or homoclinic. By Lemma~\ref{region}, the boundary of
any heteroclinic canonical strip $\Omega$ consists of (apart from
$\0$) two heteroclinic separatrices (or just $\Sigma$, when
$\Omega=R\setminus \Sigma$ is the union set of all heteroclinic
orbits except $\Sigma$) and several (possibly zero) maximal
homoclinic separatrices (Remark~\ref{maxhomo}), when $\Omega$ is
called \emph{elementary} if and only if this last set is empty.
Likewise, the boundary of a homoclinic canonical strip $\Omega$
consists of, apart from $\0$, a homoclinic separatrix $\Gamma$
enclosing it and possibly some others, all of them less than
$\Gamma$ in the $\prec$-ordering, when we again call $\Omega$
\emph{elementary} if this last family is empty. Note that is quite
possible for a canonical strip to be elementary, but at least one
heteroclinic canonical strip cannot be elementary (otherwise
$\Phi$ would have no homoclinic separatrices, and consequently all
its regular orbits would be heteroclinic, contradicting
Proposition~\ref{atractorestable}(i)).

\begin{remark}
 \label{elementary}
 The following statements are easy to prove:
 \begin{itemize}
 \item a heteroclinic canonical strip is elementary if and only if it is
  solid;
 \item a homoclininic canonical strip is elementary if and only
 if the restriction of $\Phi$ to its closure is topologically
 equivalent to the flow induced by the ``elliptic vector field''
 $f_4(x,y)= (x^2-2xy, xy-y^2)$ on the union set $A_4'$ of all
 orbits intersecting the diagonal arc $\{(x,x), 0\leq x\leq 1\}$.
 \end{itemize}
\end{remark}

\begin{remark}
 \label{ell-hyp}
If a regular homoclinic separatrix $\Gamma$ is minimal, that is,
$E(\Gamma)$ is an elementary homoclinic canonical strip, then
there is an elliptic sector intersecting $E(\Gamma)$
(Remark~\ref{elementary}). Thus, due to Remark~\ref{FSDP}, if $\0$
is not positively stable, and the finite sectorial decomposition
property holds, then the decomposition must include both an
elliptic and a hyperbolic sector.
\end{remark}

There are two natural ways to associate an orbit from
$\mathcal{X}$ to each canonical strip $\Omega$ of $\Phi$. Firstly,
$\gamma'(\Omega)$ will denote the orbit from $\mathcal{X}$
included in $\Omega$. Next, $\gamma(\Omega)$ will denote (when
$\Omega$ is homoclinic) the separatrix $\Gamma\subset\Bd\Omega$
enclosing $\Omega$,  and (when $\Omega$ is heteroclinic) the
heteroclinic separatrix $\Gamma\subset\Bd\Omega$ whose
$\omega$-point $w$ (in $C$, and then in $A$) satisfies $v<w$, $v$
being the $\omega$-point of $\gamma'(\Omega)$.
Note that $\mathcal{X}$ consists of all orbits
$\gamma(\Omega),\gamma'(\Omega)$ together with $\0$. Also, observe
that $\gamma'(\Omega)$ decomposes $\Omega$ into two components
$\Omega_l$ and $\Omega_u$, $\Omega_u$ being the component of
$\Omega\setminus \gamma'(\Omega)$ including $\gamma(\Omega)$ in
its boundary (an ambiguity arises in the case $\Omega=R\setminus
\Sigma$, where $\Omega_u$ consists of the orbits whose
$\omega$-points are greater than the $\omega$-point of
$\gamma'(\Omega)$).

\begin{lemma}
 \label{paridad}
 Let $\Omega$ be a canonical strip and let
 $\Gamma$ be a regular orbit in $\Bd\Omega$. Then $\Gamma$ has
 the same parity as $\gamma'(\Omega)$ if and only if either
 $\Gamma=\gamma(\Omega)$ or $\Gamma\in \Bd\Omega_l$.
\end{lemma}

\begin{proof}
We present the proof under the hypothesis that $\Omega$ is a
heteroclinic strip whose boundary includes two heteroclinic
orbits, $\gamma(\Omega)$ and $\gamma''(\Omega)$. The case when
$\Omega$ is heteroclinic but $\Sigma$ is the only heteroclinic
separatrix of $\Phi$, and the homoclinic case, can be dealt with
in analogous fashion. We will also assume that the fixed
orientation $\Theta$ is counterclockwise so the even (respectively
odd) homoclinic orbits coincide with the positive (respectively
negative) ones.

If $\Omega$ is elementary, then there is nothing to prove: both
$\Gamma$ and $\gamma'(\Omega)$ are heteroclinic and consequently
even. Otherwise, let $\Gamma_1, \ldots, \Gamma_j$ be the maximal
homoclinic orbits in $\Bd \Omega$, where these orbits are labelled
in such a way that if $q_1, \ldots, q_j$ are the corresponding
$\omega$-points, then $q_1 < \cdots < q_j$ (in $A$). The
corresponding $\alpha$-points will be denoted by $p_k$, $1\leq
k\leq j$. Finally, let $u$, $v$ and $w$ be the $\omega$-points of
$\gamma''(\Omega)$, $\gamma'(\Omega)$ and $\gamma(\Omega)$,
respectively (so $u<v<w$). We can assume without loss of
generality that there are small subarcs of $C$, neighbouring all
these points, which are transversal to the flow.

Let $1 \leq k \leq j-1$. We claim that it is not possible that
$\Gamma_k$ is negative and $\Gamma_{k+1}$ is positive. Assume by
contradiction $q_k<p_k<p_{k+1}<q_{k+1}$. Find points
$p_k<b<b'<p_{k+1}$ in $C$, very close to $p_k$ and $p_{k+1}$,
respectively, so that $T=\{t\in C:p_k\leq t\leq b\}$, $T'=\{t\in
C:b'\leq t\leq p_{k+1}\}$ are transversal to the flow. Also, let
$Q=\{t\in C:b\leq t\leq b'\}$. Since $\Gamma_{k}$ is negative,
backward semi-orbits starting from points from $T\setminus
\{p_k\}$ enter the disk $D$ enclosed by $C$ and, since
$\Gamma_{k+1}$ is positive, then they escape from the disk through
$Q$. Accordingly, take a a decreasing sequence
$(b_n)_{n=1}^\infty$ in $T\cap \Omega$ tending to $p_k$ and find
maximal semi-orbits $\varphi(a_n,b_n)$ fully included in $D$, when
observe that the sequence $(a_n)_n$, besides lying in $Q$, is
increasing. Call $a^*$ its limit. Clearly, $a^*\in \Cl\Omega$.
Since the full forward orbit $\varphi(a^*,+)$ lies in $D$, and
$\Gamma_k$ and $\Gamma_{k+1}$ are consecutive, we easily get that,
in fact, $a^*\in \Omega$ and there is a solid heteroclinic strip
$S$ neighbouring $a^*$. This is impossible because points $b_n$
belong to $S$ if $n$ is large enough, hence $\Gamma_k\subset \Bd
S$.

Further, if $\Gamma_k$ and $\Gamma_{k+1}$ have the same sign, then
$\gamma'(\Omega)$ cannot lie between them. In fact, assume, say,
$q_k<p_k<v<q_{k+1}<p_{k+1}$,  take $b$, $T$ and $(b_n)_n$ as
before but consider now $Q=\{t\in C:b\leq t\leq v\}$. Find
similarly the points $a_n$ and $a^*$ in $Q$ to obtain the
analogous contradiction. We prove that if $\Gamma_1$ is positive,
then $\gamma'(\Omega)$ cannot lie between $\gamma''(\Omega)$ and
$\Gamma_1$, and if $\Gamma_j$ is negative, then $\gamma'(\Omega)$
cannot lie between $\Gamma_j$ and $\gamma(\Omega)$, in the same
way.

As a conclusion, we get that either (a) all orbits $\Gamma_k$ are
positive and $\gamma'(\Omega)$ lies between $\Gamma_j$ and
$\gamma(\Omega)$, or (b) all orbits $\Gamma_k$ are negative and
$\gamma'(\Omega)$ lies between $\gamma''(\Omega)$ and $\Gamma_1$,
or (c) there is $1\leq m\leq j-1$ such that all orbits $\Gamma_k$
with $k\leq m$ are positive, all orbits with $k>m$ are negative,
and $\gamma'(\Omega)$ lies between $\Gamma_m$ and $\Gamma_{m+1}$.
This implies the lemma.

\end{proof}

We say that a finite, non-empty set $V$ of vectors of positive
integers is \emph{complete} when, for any $(i_1,\ldots,i_l)\in V$,
we have $(i_1,\ldots,i_m)\in V$ for every $1\leq m\leq l$, and
$(i_1,\ldots,i_{l-1}, i)\in V$ for every  $1\leq i\leq i_l$. If
$v\in V$, then we denote by $\lambda(v)$  the largest number $j$
such that $(v,j)\in V$, $\lambda(v)=0$ meaning that there is no
$j$ such that $(v,j)\in V$. Likewise, $\lambda(\emptyset)$ stands
for the largest number $t$ such that $(t)\in V$. Of course we
should write $\lambda_V$ instead of $\lambda$ (and similarly
$\rho_L,\sigma_L$ instead of $\rho,\sigma$ below) to emphasize
that this map depends on $V$, but hopefully this will not lead to
confusion.

Let $\mathbb{M}=\{n/3: n=0,1,2,\ldots\}$.

\begin{definition}
 \label{feasible}
 We say that a set $L$ of vectors of numbers from
 $\mathbb{M}$ is \emph{feasible} with  \emph{base} a complete set
 $V$ if its elements have the structure $(v,k)$, with $v\in V$ and
 $k\in \mathbb{M}$, and the following conditions hold:
 \begin{itemize}
  \item[(i)] for each $(i)\in V$ of length 1 there are exactly two elements in
   $L$: $(i,\lambda(i)+1$) and $(i,s+2/3)$ for some integer
   $s=\sigma(i)$, $0\leq s\leq \lambda(i)$;
  \item[(ii)] for each $v\in V$ of length at least 2 there are exactly
   four elements in $L$: $(v,0)$, $(v,\lambda(v)+1)$,
   and $(v,r+1/3)$, $(v,s+2/3)$ for some integers
   $r=\rho(v)$, $s=\sigma(v)$, $0\leq r\leq s\leq \lambda(v)$;
  \item[(iii)] $(i,\lambda(i)+2/3)$ and $(i+1,2/3)$ cannot
   simultaneously belong to $L$ (where we mean $i+1=1$ when
   $i=\lambda(\emptyset)$);
  \item[(iv)] if $\lambda(v)=1$, then
   $(v,1/3)$, $(v,5/3)$, $(v,1,1/3)$ and
   $(v,1,\lambda(v,1)+2/3)$ cannot simultaneously belong to $L$.
 \end{itemize}
\end{definition}

Note that property (iii) above implies that $\lambda(i)\geq 1$ for
some $i$, hence $V$ contains at least one sequence of length $2$.
If $V$ is the base of a feasible set $L$, then we assign a parity
(even or odd) to each $v\in V$ as follows. All vectors of length 1
in $V$ have parity even. If $(i)\in V$, then we assign even or odd
parity to $(i,j)$ depending on whether $j\leq \sigma(i)$ or not.
Inductively, once the parity of $v\in V$ is established, we assign
to $(v,j)$ the same parity as $v$, or the other one, depending on
whether $\rho(v)<j\leq \sigma(v)$ or not. Finally, if $w=(v,h)\in
L$, then we say that $w$ is an \emph{$\alpha$-vector} if either
$v$ is even and $h=0$ or $h=\rho(v)+1/3$, or $v$ is odd and
$h=\lambda(v)+1$ or $h=\sigma(v)+2/3$. Otherwise, we say that $w$
is a \emph{$\omega$-vector}.

We next explain how to associate canonically a feasible set $L$ to
$\Phi$. To construct the base $V$ we proceed inductively,
biunivocally associating to each canonical strip $\Omega$ (and the
$\omega$-point of $\gamma(\Omega)$) a vector from $V$. First of
all, order the heteroclinic canonical strips of $\Phi$ as
$\Omega_1,\ldots, \Omega_t$, this meaning that the corresponding
$\omega$-points $q_i$ of the orbits $\gamma(\Omega_i)$, $1\leq
i\leq t$, satisfy $q_1<\ldots<q_t$. Then the $1$-length vectors
from $V$ will be those of the type $(i)$, $1\leq i\leq t$. If,
additionally, the strip $\Omega_i$ is not elementary, and
$\Omega_{i,1},\dots,\Omega_{i,j}$ are the homoclinic canonical
strips $\Omega$ such that $\gamma(\Omega)\subset \Bd \Omega_i$
(again assuming $q_{i,1}<\ldots<q_{i,j}$ for their corresponding
$\omega$-points), then we add the $2$-vectors $(i,k)$ to $V$,
$1\leq k\leq j$. In general, if a vector $v$ has been added to
$V$, with corresponding canonical strip $\Omega_v$, and $\Omega_v$
is not elementary, then we consider as before the homoclinic
canonical strips $\Omega$ such that $\gamma(\Omega)\subset \Bd
\Omega_v$, call them $\Omega_{v,1},\dots,\Omega_{v,j'}$ (so that
$q_{v,1}<\ldots<q_{v,j'}$ for the corresponding $\omega$-points),
and add the vectors $(v,k)$, $1\leq k\leq j'$, to $V$. Clearly,
the set $V$ so defined is complete.

Now we define $L$ (and biunivocally associate to its vectors all
points from $X$). We just must explain how to choose the numbers
$\sigma(i)$ and the pairs $\rho(v),\sigma(v)$ in
Definition~\ref{feasible}(i) and (ii), and then check that (iii)
and (iv) hold. As for the first numbers, let (with the notation of
the previous paragraph) $1\leq i\leq t$. Then $s=\sigma(i)$ is the
largest number such that $q_{i,s}<y_i$, $y_i$ being the
$\omega$-point of $\gamma'(\Omega_i)$ (or $s=0$ if $\Omega_i$ is
elementary or no such number exists, that is, $y_i<q_{i,j}$ for
all $j$). Also, we redefine the points $y_i$ and $q_i$ as
$c_{i,\sigma(i)+2/3}$ and $c_{i,\lambda(i)+1}$, respectively. In
the general case we denote by $x_v$ and $y_v$ the $\alpha$- and
$\omega$-points of $\gamma'(\Omega_v)$ when this orbit is even,
reversing the notation when $\gamma'(\Omega_v)$ is odd, and take
$r=\rho(v)$ and $s=\sigma(v)$ as the largest numbers satisfying
$q_{v,r}<x_v$ and $q_{v,s}<y_v$, respectively (or $r=s=0$ when
$\Omega_v$ is elementary, and $r=0$ or $s=0$ when the
corresponding number does not exist). Finally, we redenote $x_v$
and $y_v$ as $c_{v,\rho(v)+1/3}$ and $c_{v,\sigma(v)+2/3}$, while
$c_{v,0}$ and $c_{v,\lambda(v)+1}$ stand for the $\alpha$- and
$\omega$-points (or conversely in the odd case) of
$\gamma(\Omega_v)$.

We claim that (iii) in Definition~\ref{feasible} holds. Indeed if,
say, both $(i,\lambda(i)+2/3)$ and $(i+1,2/3)$ belong to $L$ for
some $i$, the orbits $\gamma'(\Omega_i)$ and
$\gamma'(\Omega_{i+1})$ would bound, together with $\0$, a solid
strip (Remark~\ref{problema}). Since this strip includes the
separatrix $\gamma(\Omega_i)$, we get a contradiction.

Assume now that Definition~\ref{feasible}(iv) does not hold, that
is, there is $v\in V$ with $\lambda(v)=1$ such that all vectors
$(v,1/3)$, $(v,5/3)$, $(v,1,1/3)$ and $(v,1,\lambda(v,1)+2/3)$
belong to $L$. Then, again by Remark~\ref{problema}, the orbits
$\gamma'(\Omega_v)$, $\gamma'(\Omega_{v,1})$ bound, together with
$\0$, a solid strip including $\gamma(\Omega_{v,1})$, which is
impossible.

Thus we have shown that $L$ is feasible. Although $L$  has been
constructed with the help of the circle $C$, it depends only on
$\Theta$ and $\Sigma$. We call it the \emph{canonical feasible
set} associated to $\Phi$, the orientation $\Theta$ and the
separatrix $\Sigma$.

As some examples, we present in Tables~\ref{tableEx1}
and~\ref{tableEx2} the feasible sets associated to the flows on
Figure~\ref{FSDnobasta} under the counterclockwise orientation.

\begin{table}
 \renewcommand{\arraystretch}{1.5}
 \centering
\begin{tabu}{c|[1pt]c}
 $V$ & $L$ \\
 \tabucline[1pt]{-}
 $(1)$ & $(1,2)$, $(1,\frac{5}{3})$  \\
 \hline
 $(1,1)$ & $(1,1,0)$, $(1,1,2)$, $(1,1,\frac{1}{3})$, $(1,1,\frac{2}{3})$\\
 \hline
 $(1,1,1)$ & $(1,1,1,0)$, $(1,1,1,1)$, $(1,1,1,\frac{1}{3})$, $(1,1,1,\frac{2}{3})$\\
 \hline
\end{tabu}
\caption{\label{tableEx1} The elements of the feasible set $L$ and
its base $V$ from the left flow of Figure~\ref{FSDnobasta}.}
\end{table}

\begin{table}
 \renewcommand{\arraystretch}{1.5}
 \centering
\begin{tabu}{c|[1pt]c}
 $V$ & $L$ \\
 \tabucline[1pt]{-}
 $(1)$ & $(1,2)$, $(1,\frac{2}{3})$\\
 \hline
 $(1,1)$ & $(1,1,0)$, $(1,1,1)$, $(1,1,\frac{1}{3})$, $(1,1,\frac{2}{3})$\\
 \hline
 $(2)$ & $(2,1)$, $(2,\frac{2}{3})$\\
 \hline
 $(3)$ & $(3,2)$, $(3,\frac{5}{3})$\\
 \hline
 $(3,1)$ & $(3,1,0)$, $(3,1,1)$, $(3,1,\frac{1}{3})$, $(3,1,\frac{2}{3})$\\
 \hline
 \end{tabu}
\caption{\label{tableEx2} The elements of the feasible set $L$ and
its base $V$ from the right flow of Figure~\ref{FSDnobasta}
($\Sigma$ is the ``upper'' heteroclinic separatrix).}
\end{table}

\begin{remark}
 \label{ejemplofacil}
 The simplest feasible set
  $$
  L=\{(1,5/3),(1,2),(1,1,0),(1,1,1/3),(1,1,2/3),(1,1,1)\}
  $$
 (equivalent to
  $$
  L=\{(1,2/3),(1,2),(1,1,0),(1,1,1/3),(1,1,2/3),(1,1,1)\}
  $$
 after reversing the orientation) correspond to the case when
 there are exactly three separatrices (one heteroclinic,  another
 one regular homoclinic, and the singular point), which  occurs
 when ``$\prec$'' is a total ordering in $\mathcal{H}$ ($\0$
 becoming an elliptic saddle for the flow).
\end{remark}

Observe that the bijection from $L$ to $X$ given by $w\mapsto c_w$
preserves \emph{orders} (when the lexicographical order is used in
$L$), \emph{orbits} (that is, two points $c_w$ and $c_{w'}$
belongs to the same orbit if and only if $w=(v,h)$ and $w'=(v,h')$
for some $v\in V$  and $h+h'$ is an integer) and \emph{directions}
(that is, $w$ is a $\omega$-vector if and only if $c_w$ is a
$\omega$-point; this follows from Lemma~\ref{paridad}, which
implies that the parity of $v\in V$ is the same as that of
$\gamma(\Omega_v)$ and $\gamma'(\Omega_v)$). There are many
feasible sets $L'$ which can be bijectively mapped onto $X$ so
that ordering is preserving: since both orderings are total, one
just needs that both cardinalities of $L$ and $L'$ are the same.
As it turns out, if orbits are preserved, then directions are
preserved as well:

\begin{lemma}
 \label{redundante}
 If $L'$ is feasible, and there is a bijection $\psi:L'\to X$
 preserving orders and orbits, then $L'=L$.
\end{lemma}

\begin{proof}
Let $V'$ the base of $L'$ and redenote $\lambda_{V'}=\lambda'$,
$\rho_{L'}=\rho'$, $\sigma_{L'}=\sigma'$. Since $\psi$ preserves
orbits, it maps vectors $(i',\lambda'(i')+1)$ and
$(i',\sigma'(i')+2/3)$  to $\omega$-points of heteroclinic orbits,
and pairs $(v',0)$ and $(v',\lambda'(v')+1)$, as well as pairs
$(v',\rho'(v')+1/3)$ and $(v',\sigma'(v')+2/3)$, to pairs of
points of homoclinic orbits. Since orders are preserved as well,
we get that vectors $(i',\lambda'(i')+1)$ are precisely those
mapped to heteroclinic separatrices, and deduce that vectors of
lengths 1 and 2 of $V$ and $V'$, as well as vectors of length 2 of
$L$ and $L'$, are the same. Now, as the reader will easily
convince himself, to prove the lemma we just have to show this:
pairs $(v',0)$ and $(v',\lambda'(v')+1)$ are exactly those mapped
to homoclinic separatrices.

Assume, to arrive at a contradiction, that $(v',0)$ and
$(v',\lambda'(i')+1)$ are mapped to one of the orbits
$\gamma'(\Omega_v)$ of $\mathcal{X}$. Since $\mathcal{X}$ has no
orbits between $\gamma'(\Omega_v)$ and the orbits
$\gamma(\Omega_{v,k})$  (regarding the order ``$\prec$''), it is
clear that $(v',\rho'(v')+1/3)$ and $(v',\sigma'(v')+2/3)$ must be
mapped to one of the orbits $\gamma(\Omega_{v,k})$ (in particular,
$v$ cannot have maximal length in $V$). Similarly, if
$(v',\rho'(v')+1/3)$ and $(v',\sigma'(v')+2/3)$ are mapped to an
orbit $\gamma(\Omega_w)$, the pair which is mapped to
$\gamma'(\Omega_w)$ must be of the type $(w',0)$ and
$(w',\lambda'(w')+1)$, because the orbit corresponding to
$(w',\rho'(w')+1/3)$ and $(w',\sigma'(w')+2/3)$ is $\prec$-less
than that corresponding to $(w',0)$ and $(w',\lambda'(w')+1)$, and
there are no orbits of $\mathcal{X}$ between $\gamma(\Omega_w)$
and $\gamma'(\Omega_w)$. We could thus proceed indefinitely,
contradicting the finiteness of $\mathcal{X}$.
\end{proof}

\begin{proof}[Proof of Theorem~\ref{principal}]
 The statement (i)$\Rightarrow$(ii) is obvious (recall
Proposition~\ref{atractorestable}).

 Let us show (ii)$\Rightarrow$(iii). Fix small circles $C,C'$ around
 $\0$ and let $\mu:C\to C'$ be a homeomorphism preserving
 orbits for $\xi$. Use the hypothesis to find
 heteroclinic separatrices $\Sigma$ and $\Sigma'$ such that
 $\xi(\Sigma)=\Sigma'$, fix $\Theta$ as the counterclockwise
 orientation, and take $\Theta'$ as the counterclockwise or the
 clockwise orientation depending on whether $\mu$ preserves or
 reverses the orientation. Construct the canonical
 feasible sets $L$ and $L'$ associated to them,
 and the corresponding bijections $\psi:L\to X$, $\psi':L'\to X'$ to
 the configurations of $\mathcal{X}$ and $\mathcal{X'}$ preserving
 orders, orbits and directions. Although the hypothesis does not
 state that $\mu$ preserves directions for $\xi$,
 we get that $\mu^{-1}\circ \psi':L'\to X$ preserves orders and
 orbits anyway. Now Lemma~\ref{redundante} applies and (iii)
 follows.

 Finally, to prove (iii)$\Rightarrow$(i), let
 again $C,C'$ be small circles around $\0$, denote the
 configurations of $\mathcal{X}$ and $\mathcal{X'}$ in these
 circles by $X$ and $X'$, and find arcs $A\subset C$ and $A'\subset
 C'$ containing all points of $X$ and $X'$ and having $q$
 and $q'$, the $\omega$-points of $\Sigma$ and $\Sigma'$, as their
 upper endpoints (after using the respective orientations $\Theta$
 and $\Theta'$). According to the hypothesis, there are a feasible
 set $L$ and bijections $\psi:L\to X$, $\psi':L\to X'$ preserving
 orders, orbits and directions, and hence a bijection $\xi:\mathcal{X}\to
 \mathcal{X'}$ and a homeomorphism $\mu:C\to C'$ preserving
 orbits and directions for $\xi$. Then, as explained in
 Subsection~\ref{extensiones}, there is a plane homeomorphism
 preserving the skeletons orbits, which
 turns out to preserve the flows directions as well. Hence
 $\Phi$ and $\Phi'$ are topologically
 equivalent by Theorem~\ref{markus}.
\end{proof}

%
%
%
%

\section{Proof of Theorem~\ref{ejemploanalitico}}

Let $0\leq s\leq j$ be non-negative integers. We define a
$C^1$-vector field $f_{s,j}$ as follows. We start from
$f(x,y)=(x(x^2-1),-y)$. As easily checked, the phase portrait of
(the associated local flow to) $f$ in the semi-band $[-1,1]\times
[0,\infty)$ (the only sector we are interested in) consists of
three singular points, the attracting node $\0$ and the saddle
points $(-1,0)$ and $(1,0)$, two horizontal orbits in the $x$-axis
going to $\0$ as time goes to $\infty$, and three vertical orbits
on the semi-lines $x=-1,0,-1$, each converging in positive time to
the corresponding singular point. All other orbits go to $\0$ as
$t\to\infty$. Next, let $\kappa(x)$ be a non-negative
$C^1$-function vanishing at points $x=-i/s$, $0\leq i\leq s$ (or
at the whole interval $[-1,0]$ if $s=0$), at points $x=i/(j-s)$,
$0\leq i\leq j-s$ (or at the whole interval $[0,1]$ if $s=j$), and
at no other points. Then we define
$f_{s,j}(x,y)=(\kappa(x)+y^2)f(x,y)$, thus adding new singular
points in the $x$-axis and leaving unchanged the upper orbits.
Figure~\ref{campof} exhibits the phase portrait of $f_{s,j}$ for
different values of $s$ and $j$.

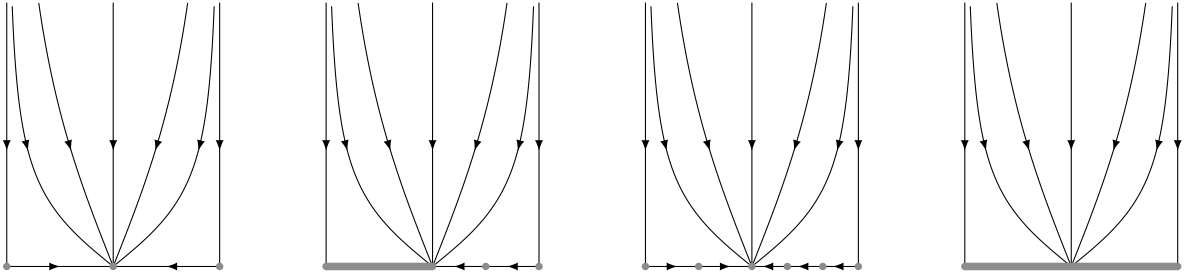
\begin{figure}
\centering
\begin{tikzpicture}[scale=1.4]

 \draw (-1.0,0.0) -- (1.0,0.0);
 \draw (-1.0,0.0) -- (-1.0,2.5);
 \draw (1.0,0.0) -- (1.0,2.5);
 \draw (0.0,0.0) -- (0.0,2.5);

 \draw[latex-] (-1.0, 1.09561) -- (-1.0, 1.13498);
 \draw[latex-] (1.0, 1.09561) -- (1.0, 1.13498);
 \draw[latex-] (0.0, 1.09561) -- (0.0, 1.13498);
 \draw[latex-] (-0.5, 0.0) -- (-0.51, 0.0);
 \draw[latex-] (0.5, 0.0) -- (0.51, 0.0);


    \draw[smooth,samples=100,domain=0.0:0.95] plot(\x,{((2.5 * sqrt(1 - (0.95)^2))/0.95) * (\x/sqrt(1 - \x^2))});
    \draw[latex-] (0.8, 1.09561) -- (0.81, 1.13498);

    \draw[smooth,samples=100,domain=0.0:0.7] plot(\x,{((2.5 * sqrt(1 - (0.7)^2))/0.7) * (\x/sqrt(1 - \x^2))});

    \draw[latex-] (0.3947, 1.09564) -- (0.4047, 1.12876) ;

    \draw[smooth,samples=100,domain=0.0:0.7] plot(-\x,{((2.5 * sqrt(1 - (0.7)^2))/0.7) * (\x/sqrt(1 - \x^2))});

    \draw[latex-] (-0.3947, 1.09564) -- (-0.4047, 1.12876) ;

    \draw[smooth,samples=100,domain=0.0:0.95] plot(-\x,{((2.5 * sqrt(1 - (0.95)^2))/0.95) * (\x/sqrt(1 - \x^2))});
    \draw[latex-] (-0.8, 1.09561) -- (-0.81, 1.13498);

    \fill[color=gray!110!black] (0.0,0.0) circle (0.035);
    \fill[color=gray!110!black] (1.0,0.0) circle (0.035);
    \fill[color=gray!110!black] (-1.0,0.0) circle (0.035);


 \draw[xshift=3cm] (-1.0,0.0) -- (1.0,0.0);
 \draw[xshift=3cm] (-1.0,0.0) -- (-1.0,2.5);
 \draw[xshift=3cm] (1.0,0.0) -- (1.0,2.5);
 \draw[xshift=3cm] (0.0,0.0) -- (0.0,2.5);

 \draw[xshift=3cm, latex-] (-1.0, 1.09561) -- (-1.0, 1.13498);
 \draw[xshift=3cm, latex-] (1.0, 1.09561) -- (1.0, 1.13498);
 \draw[xshift=3cm, latex-] (0.0, 1.09561) -- (0.0, 1.13498);
 \draw[xshift=3cm, latex-] (0.2, 0.0) -- (0.21, 0.0);
 \draw[xshift=3cm, latex-] (0.7, 0.0) -- (0.71, 0.0);


 \draw[xshift=3cm   , smooth,samples=100,domain=0.0:0.95] plot(\x,{((2.5 * sqrt(1 - (0.95)^2))/0.95) * (\x/sqrt(1 - \x^2))});
 \draw[xshift=3cm   , latex-] (0.8, 1.09561) -- (0.81, 1.13498);
 \draw[xshift=3cm   , smooth,samples=100,domain=0.0:0.7] plot(\x,{((2.5 * sqrt(1 - (0.7)^2))/0.7) * (\x/sqrt(1 - \x^2))});

    \draw[xshift=3cm   , latex-] (0.3947, 1.09564) -- (0.4047, 1.12876) ;

    \draw[xshift=3cm   , smooth,samples=100,domain=0.0:0.7] plot(-\x,{((2.5 * sqrt(1 - (0.7)^2))/0.7) * (\x/sqrt(1 - \x^2))});

    \draw[xshift=3cm   , latex-] (-0.3947, 1.09564) -- (-0.4047, 1.12876) ;

    \draw[xshift=3cm   , smooth,samples=100,domain=0.0:0.95] plot(-\x,{((2.5 * sqrt(1 - (0.95)^2))/0.95) * (\x/sqrt(1 - \x^2))});
    \draw[xshift=3cm   , latex-] (-0.8, 1.09561) -- (-0.81, 1.13498);

 \draw[xshift=3cm, color=gray!110!black,line width=0.1cm] (-1.0,0.0) -- (0.0,0.0);
 
 \fill[xshift=3cm, color=gray!110!black] (-1.0,0.0) circle (0.035cm);
 \fill[xshift=3cm, color=gray!110!black] (0.0,0.0) circle (0.035);
 \fill[xshift=3cm, color=gray!110!black] (1/2,0.0) circle (0.035);
 \fill[xshift=3cm, color=gray!110!black] (1.0,0.0) circle (0.035);


 \draw[xshift=6cm   ] (-1.0,0.0) -- (1.0,0.0);
 \draw[xshift=6cm   ] (-1.0,0.0) -- (-1.0,2.5);
 \draw[xshift=6cm   ] (1.0,0.0) -- (1.0,2.5);
 \draw[xshift=6cm   ] (0.0,0.0) -- (0.0,2.5);

 \draw[xshift=6cm, latex-] (-1.0, 1.09561) -- (-1.0, 1.13498);
 \draw[xshift=6cm, latex-] (1.0, 1.09561) -- (1.0, 1.13498);
 \draw[xshift=6cm, latex-] (0.0, 1.09561) -- (0.0, 1.13498);

 \draw[xshift=6cm, latex-] (-0.2, 0.0) -- (-0.21, 0.0);
 \draw[xshift=6cm, latex-] (-0.7, 0.0) -- (-0.71, 0.0);
 \draw[xshift=6cm, latex-] (0.1, 0.0) -- (0.11, 0.0);
 \draw[xshift=6cm, latex-] (0.43, 0.0) -- (0.44, 0.0);
 \draw[xshift=6cm, latex-] (0.76, 0.0) -- (0.77, 0.0);


    \draw[xshift=6cm   , smooth,samples=100,domain=0.0:0.95] plot(\x,{((2.5 * sqrt(1 - (0.95)^2))/0.95) * (\x/sqrt(1 - \x^2))});
    \draw[xshift=6cm   , latex-] (0.8, 1.09561) -- (0.81, 1.13498);

    \draw[xshift=6cm   , smooth,samples=100,domain=0.0:0.7] plot(\x,{((2.5 * sqrt(1 - (0.7)^2))/0.7) * (\x/sqrt(1 - \x^2))});

    \draw[xshift=6cm   , latex-] (0.3947, 1.09564) -- (0.4047, 1.12876) ;

    \draw[xshift=6cm   , smooth,samples=100,domain=0.0:0.7] plot(-\x,{((2.5 * sqrt(1 - (0.7)^2))/0.7) * (\x/sqrt(1 - \x^2))});

    \draw[xshift=6cm   , latex-] (-0.3947, 1.09564) -- (-0.4047, 1.12876) ;

    \draw[xshift=6cm   , smooth,samples=100,domain=0.0:0.95] plot(-\x,{((2.5 * sqrt(1 - (0.95)^2))/0.95) * (\x/sqrt(1 - \x^2))});
    \draw[xshift=6cm   , latex-] (-0.8, 1.09561) -- (-0.81, 1.13498);

 \fill[xshift=6cm, color=gray!110!black] (-1.0,0.0) circle (0.035);
 \fill[xshift=6cm, color=gray!110!black] (-0.5,0.0) circle (0.035);
 \fill[xshift=6cm, color=gray!110!black] (0.0,0.0) circle (0.035);
 \fill[xshift=6cm, color=gray!110!black] (0.333,0.0) circle (0.035);
 \fill[xshift=6cm, color=gray!110!black] (0.666,0.0) circle (0.035);
 \fill[xshift=6cm, color=gray!110!black] (1.0,0.0) circle (0.035);


 \draw[xshift=9cm   ] (-1.0,0.0) -- (1.0,0.0);
 \draw[xshift=9cm   ] (-1.0,0.0) -- (-1.0,2.5);
 \draw[xshift=9cm   ] (1.0,0.0) -- (1.0,2.5);
 \draw[xshift=9cm   ] (0.0,0.0) -- (0.0,2.5);

 \draw[xshift=9cm, latex-] (-1.0, 1.09561) -- (-1.0, 1.13498);
 \draw[xshift=9cm, latex-] (1.0, 1.09561) -- (1.0, 1.13498);
 \draw[xshift=9cm, latex-] (0.0, 1.09561) -- (0.0, 1.13498);


    \draw[xshift=9cm   , smooth,samples=100,domain=0.0:0.95] plot(\x,{((2.5 * sqrt(1 - (0.95)^2))/0.95) * (\x/sqrt(1 - \x^2))});
    \draw[xshift=9cm   , latex-] (0.8, 1.09561) -- (0.81, 1.13498);

    \draw[xshift=9cm   , smooth,samples=100,domain=0.0:0.7] plot(\x,{((2.5 * sqrt(1 - (0.7)^2))/0.7) * (\x/sqrt(1 - \x^2))});

    \draw[xshift=9cm   , latex-] (0.3947, 1.09564) -- (0.4047, 1.12876) ;

    \draw[xshift=9cm, smooth,samples=100,domain=0.0:0.7] plot(-\x,{((2.5 * sqrt(1 - (0.7)^2))/0.7) * (\x/sqrt(1 - \x^2))});

    \draw[xshift=9cm, latex-] (-0.3947, 1.09564) -- (-0.4047, 1.12876) ;

    \draw[xshift=9cm, smooth,samples=100,domain=0.0:0.95] plot(-\x,{((2.5 * sqrt(1 - (0.95)^2))/0.95) * (\x/sqrt(1 - \x^2))});
    \draw[xshift=9cm, latex-] (-0.8, 1.09561) -- (-0.81, 1.13498);

\draw[xshift=9cm,color=gray!110!black,line width=0.1cm] (-1.0,0.0) -- (1.0,0.0);
  
    \fill[xshift=9cm,color=gray!110!black] (1.0,0.0) circle (0.035);
    \fill[xshift=9cm,color=gray!110!black] (-1.0,0.0) circle (0.035);

\end{tikzpicture}
\caption{From left to right: phase portraits of $f$ (and
$f_{1,2}$), $f_{0,2}$, $f_{2,5}$ and $f_{0,0}$. \label{campof}}
\end{figure}

Now, let $0\leq r\leq s\leq j$ be non-negative integers and define
$C^1$-vector fields $g_{r,s,j}^+,g_{r,s,j}^-$ as follows. This
time our starting point is
 $$
 g(x,y)=\left((x^2-1)\left(x^2-\left(1-\frac{(1-y)^2}{2}\right)^2\right),y(y-1)x\right)
 $$
and we are interested in its phase portrait in the rectangle
$[-1,1]\times [0,1]$. We have six singular points: the saddles
$(-1,0)$ and $(1,0)$, the repelling node $(-1/2,0)$, the
attracting node $(1/2,0)$ and the semi-hyperbolic singularities
$(-1,1)$ and $(1,1)$. The boundary of the rectangle is invariant
for the flow, hence consisting of the singular points and six
regular orbits, all clockwise oriented by the flow except that
connecting $(-1/2,0)$ and $(1/2,0)$. Additional isoclines exist at
the $y$-axis (for the horizontal direction of the flow) and the
parabolas $x=\pm (1-(1-y)^2/2)$ (for the vertical direction of the
flow), which ensures that all orbits in the rectangle interior
crossing the $y$-axis go to $(-1/2,0)$ (respectively, $(1/2,0)$)
as time goes to $-\infty$ (respectively, $\infty$). See
Figure~\ref{campoginicial}.

\begin{figure}
\centering
\begin{tikzpicture}[scale=4.0]

\draw[fill=gray!40!white, draw=gray!20!white] plot[color=white,
smooth,samples=100,domain=-1.0:1.0] (\x,{0}) --
    plot[color=white,smooth,samples=100,domain=1.0:-1.0]  (\x,{1});

\draw[fill=gray!15!white, draw=gray!20!white] plot[color=white,
smooth,samples=100,domain=0.50:1.0] (\x,{1 - sqrt(2*(1 - \x))}) --
    plot[color=white,smooth,samples=100,domain=1.0:0.50]  (\x,{0});
\draw[fill=gray!5!white, draw=gray!20!white] plot[color=white,
smooth,samples=100,domain=-0.50:-1.0] (\x,{1 - sqrt(2*(1 + \x))})
--
  plot[color=white,smooth,samples=100,domain=-1.0:-0.50]  (\x,{0});

 \draw (-1.0,0.0) -- (1.0,0.0);
 \draw (-1.0,1.0) -- (1.0,1.0);
 \draw (-1.0,0.0) -- (-1.0,1.0);
 \draw (1.0,0.0) -- (1.0,1.0);
 \draw[dashed] (0.0,0.0) -- (0.0,1.0);

\draw[smooth,samples=100,domain=0.5:1.0, dashed] plot(\x,{1 -
sqrt(2*(1 - \x))});

\draw[smooth,samples=100,domain=-1.0:-0.5, dashed] plot(\x,{1 -
sqrt(2*(1 + \x))});

 \node[color=black!70] at (0.0,0.55){\fder};
 \node[color=black!70] at (0.75,0.30){\faba};
 \node[color=black!70] at (-0.75,0.30){\farr};

 \node[color=black!70] at (0.4,0.4) [style={rectangle,draw,inner sep=1}] {\NWSE};
 \node[color=black!70] at (-0.4,0.4) [style={rectangle,draw,inner sep=1}] {\SWNE};
 \node[color=black!70] at (0.9,0.15) [style={rectangle,draw,inner sep=1}] {\nesw};
 \node[color=black!70] at (-0.9,0.15) [style={rectangle,draw,inner sep=1}] {\senw};

\draw (0., 0.9) -- (0.200613, 0.898077) -- (0.372594,
  0.892417) -- (0.505107, 0.883249) -- (0.602276,
  0.870758) -- (0.672829, 0.855012) -- (0.724555,
  0.835982) -- (0.763098, 0.813585) -- (0.792275,
  0.787734) -- (0.814619, 0.758383) -- (0.831815,
  0.725568) -- (0.844985, 0.689445) -- (0.854884,
  0.65031) -- (0.862014, 0.60862) -- (0.866707, 0.56498) -- (0.869173,
   0.520125) -- (0.869542, 0.474869) -- (0.867884,
  0.430055) -- (0.864237, 0.38649) -- (0.858623,
  0.344888) -- (0.851071, 0.305824) -- (0.841628,
  0.269715) -- (0.830381, 0.236807) -- (0.817468,
  0.20719) -- (0.803081, 0.180824) -- (0.787465,
  0.157562) -- (0.77091, 0.137191) -- (0.753732,
  0.11945) -- (0.736253, 0.104065) -- (0.718783,
  0.0907575) -- (0.701598, 0.0792654) -- (0.684934,
  0.0693453) -- (0.668978, 0.0607789) -- (0.653866,
  0.0533736) -- (0.639688, 0.0469619) -- (0.626493,
  0.0414) -- (0.614298, 0.0365646) -- (0.603091,
  0.0323511) -- (0.592843, 0.0286709) -- (0.583509,
  0.025449) -- (0.575038, 0.0226216) -- (0.567369,
  0.020135) -- (0.560444, 0.0179434) -- (0.554203,
  0.0160081) -- (0.548586, 0.0142959) -- (0.543538,
  0.0127784) -- (0.539006, 0.0114314) -- (0.53494,
  0.0102339) -- (0.531295, 0.00916808) -- (0.528029,
  0.00821815) -- (0.525104, 0.00737064) -- (0.522485,
  0.00661373) -- (0.52014, 0.00593713) -- (0.518041,
  0.00533183) -- (0.516162, 0.00478992) -- (0.514481,
  0.00430444) -- (0.512976, 0.00386925) -- (0.51163,
  0.00347893) -- (0.510424, 0.00312868) -- (0.509345,
  0.00281427) -- (0.508379, 0.00253192) -- (0.507514,
  0.00227826) -- (0.506739, 0.00205031) -- (0.506046,
  0.00184541) -- (0.505424, 0.00166117) -- (0.504867,
  0.00149549) -- (0.504368, 0.00134645) -- (0.503921,
  0.00121237) -- (0.50352, 0.00109172) -- (0.50316,
  0.000983151) -- (0.502838, 0.000885425) -- (0.502549,
  0.00079746) -- (0.50229, 0.000718265) -- (0.502057,
  0.000646966) -- (0.501848, 0.000582765) -- (0.501661,
  0.000524954) -- (0.501492, 0.000472893) -- (0.501341,
  0.000426005) -- (0.501206, 0.000383778) -- (0.501084,
  0.000345743) -- (0.500975, 0.000311483) -- (0.500876,
  0.000280624) -- (0.500037, 0.0000123131);

\draw (0., 0.9) -- (-0.200613, 0.898077) -- (-0.372594,
  0.892417) -- (-0.505107, 0.883249) -- (-0.602276,
  0.870758) -- (-0.672829, 0.855012) -- (-0.724555,
  0.835982) -- (-0.763098, 0.813585) -- (-0.792275,
  0.787734) -- (-0.814619, 0.758383) -- (-0.831815,
  0.725568) -- (-0.844985, 0.689445) -- (-0.854884,
  0.65031) -- (-0.862014, 0.60862) -- (-0.866707,
  0.56498) -- (-0.869173, 0.520125) -- (-0.869542,
  0.474869) -- (-0.867884, 0.430055) -- (-0.864237,
  0.38649) -- (-0.858623, 0.344888) -- (-0.851071,
  0.305824) -- (-0.841628, 0.269715) -- (-0.830381,
  0.236807) -- (-0.817468, 0.20719) -- (-0.803081,
  0.180824) -- (-0.787465, 0.157562) -- (-0.77091,
  0.137191) -- (-0.753732, 0.11945) -- (-0.736253,
  0.104065) -- (-0.718783, 0.0907575) -- (-0.701598,
  0.0792654) -- (-0.684934, 0.0693453) -- (-0.668978,
  0.0607789) -- (-0.653866, 0.0533736) -- (-0.639688,
  0.0469619) -- (-0.626493, 0.0414) -- (-0.614298,
  0.0365646) -- (-0.603091, 0.0323511) -- (-0.592843,
  0.0286709) -- (-0.583509, 0.025449) -- (-0.575038,
  0.0226216) -- (-0.567369, 0.020135) -- (-0.560444,
  0.0179434) -- (-0.554203, 0.0160081) -- (-0.548586,
  0.0142959) -- (-0.543538, 0.0127784) -- (-0.539006,
  0.0114314) -- (-0.53494, 0.0102339) -- (-0.531295,
  0.00916808) -- (-0.528029, 0.00821815) -- (-0.525104,
  0.00737064) -- (-0.522485, 0.00661373) -- (-0.52014,
  0.00593713) -- (-0.518041, 0.00533183) -- (-0.516162,
  0.00478992) -- (-0.514481, 0.00430444) -- (-0.512976,
  0.00386925) -- (-0.51163, 0.00347893) -- (-0.510424,
  0.00312868) -- (-0.509345, 0.00281427) -- (-0.508379,
  0.00253192) -- (-0.507514, 0.00227826) -- (-0.506739,
  0.00205031) -- (-0.506046, 0.00184541) -- (-0.505424,
  0.00166117) -- (-0.504867, 0.00149549) -- (-0.504368,
  0.00134645) -- (-0.503921, 0.00121237) -- (-0.50352,
  0.00109172) -- (-0.50316, 0.000983151) -- (-0.502838,
  0.000885425) -- (-0.502549, 0.00079746) -- (-0.50229,
  0.000718265) -- (-0.502057, 0.000646966) -- (-0.501848,
  0.000582765) -- (-0.501661, 0.000524954) -- (-0.501492,
  0.000472893) -- (-0.501341, 0.000426005) -- (-0.501206,
  0.000383778) -- (-0.501084, 0.000345743) -- (-0.500975,
  0.000311483) -- (-0.500876, 0.000280624) -- (-0.500788,
  0.000252826) -- (-0.500709, 0.000227786) -- (-0.500637,
  0.000205227) -- (-0.500573, 0.000184907) -- (-0.500516,
  0.000166598) -- (-0.500464, 0.000150106) -- (-0.500418,
  0.000135245) -- (-0.500376, 0.000121858) -- (-0.500338,
  0.000109796) -- (-0.500304, 0.0000989293) -- (-0.500274,
  0.000089138) -- (-0.500246, 0.0000803161) -- (-0.500222,
  0.0000723682) -- (-0.5002, 0.0000652065) -- (-0.50018,
  0.0000587546) -- (-0.500162, 0.0000529404) -- (-0.500146,
  0.0000477024) -- (-0.500131, 0.0000429823) -- (-0.500118,
  0.0000387295) -- (-0.500106, 0.0000348976) -- (-0.500096,
  0.0000314447) -- (-0.500086, 0.0000283338) -- (-0.500078,
  0.0000255304) -- (-0.50007, 0.0000230047) -- (-0.500063,
  0.0000207285) -- (-0.500057, 0.0000186778) -- (-0.500051,
  0.0000168304) -- (-0.500046, 0.0000151651) -- (-0.500041,
  0.0000136648) -- (-0.500037, 0.0000123131);

     \draw[latex-](0.505107, 0.883249) -- (0.372594, 0.892417);

\draw (0., 0.1) -- (-0.0731985, 0.0993133) -- (-0.143199,
  0.097308) -- (-0.20744, 0.0941365) -- (-0.264357,
  0.0900174) -- (-0.313381, 0.0851971) -- (-0.354701,
  0.079918) -- (-0.388972, 0.0743961) -- (-0.41705,
  0.0688107) -- (-0.439838, 0.0633018) -- (-0.458183,
  0.0579727) -- (-0.472837, 0.052895) -- (-0.484448,
  0.0481143) -- (-0.493562, 0.0436559) -- (-0.500635,
  0.0395296) -- (-0.506045, 0.0357341) -- (-0.510104,
  0.0322601) -- (-0.513071, 0.0290931) -- (-0.515157,
  0.0262152) -- (-0.516538, 0.0236068) -- (-0.517358,
  0.0212476) -- (-0.517735, 0.0191172) -- (-0.517766,
  0.017196) -- (-0.517532, 0.0154651) -- (-0.517097,
  0.013907) -- (-0.516514, 0.0125052) -- (-0.515825,
  0.0112445) -- (-0.515065, 0.0101112) -- (-0.514262,
  0.00909247) -- (-0.513437, 0.00817694) -- (-0.512608,
  0.00735417) -- (-0.511787, 0.00661477) -- (-0.510984,
  0.00595026) -- (-0.510207, 0.00535301) -- (-0.509462,
  0.00481621) -- (-0.508751, 0.00433363) -- (-0.508077,
  0.0038998) -- (-0.507442, 0.0035097) -- (-0.506846,
  0.00315892) -- (-0.506287, 0.00284344) -- (-0.505767,
  0.00255966) -- (-0.505283, 0.00230439) -- (-0.504834,
  0.00207471) -- (-0.504419, 0.00186806) -- (-0.504036,
  0.0016821) -- (-0.503682, 0.00151473) -- (-0.503357,
  0.00136409) -- (-0.503059, 0.0012285) -- (-0.502784,
  0.00110643) -- (-0.502533, 0.000996534) -- (-0.502303,
  0.000897589) -- (-0.502093, 0.000808498) -- (-0.501901,
  0.000728272) -- (-0.501726, 0.000656028) -- (-0.501566,
  0.000590967) -- (-0.501421, 0.00053237) -- (-0.501288,
  0.000479596) -- (-0.501168, 0.000432062) -- (-0.501058,
  0.000389246) -- (-0.500958, 0.00035068) -- (-0.500868,
  0.00031594) -- (-0.500786, 0.000284646) -- (-0.500711,
  0.000256454) -- (-0.500643, 0.000231059) -- (-0.500582,
  0.000208179) -- (-0.500526, 0.000187568) -- (-0.500476,
  0.000168998) -- (-0.50043, 0.000152268) -- (-0.500389,
  0.000137195) -- (-0.500351, 0.000123616) -- (-0.500318,
  0.00011138) -- (-0.500287, 0.000100357) -- (-0.500259,
  0.0000904237) -- (-0.500234, 0.0000814751) -- (-0.500211,
  0.0000734114) -- (-0.500191, 0.0000661468) -- (-0.500172,
  0.0000596008) -- (-0.500156, 0.0000537032) -- (-0.50014,
  0.0000483891) -- (-0.500127, 0.0000436009) -- (-0.500114,
  0.0000392867) -- (-0.500103, 0.0000353993) -- (-0.500093,
  0.0000318976) -- (-0.500084, 0.0000287421) -- (-0.500076,
  0.0000258996) -- (-0.500068, 0.0000233383) -- (-0.500062,
  0.0000210288) -- (-0.500056, 0.0000189481) -- (-0.50005,
  0.0000170734) -- (-0.500045, 0.0000153838) -- (-0.500041,
  0.0000138619) -- (-0.500037, 0.0000124905) -- (-0.500033,
  0.0000112548) -- (-0.50003, 0.0000101417);

    \draw (0., 0.1) -- (0.0731985, 0.0993133) -- (0.143199,
  0.097308) -- (0.20744, 0.0941365) -- (0.264357,
  0.0900174) -- (0.313381, 0.0851971) -- (0.354701,
  0.079918) -- (0.388972, 0.0743961) -- (0.41705,
  0.0688107) -- (0.439838, 0.0633018) -- (0.458183,
  0.0579727) -- (0.472837, 0.052895) -- (0.484448,
  0.0481143) -- (0.493562, 0.0436559) -- (0.500635,
  0.0395296) -- (0.506045, 0.0357341) -- (0.510104,
  0.0322601) -- (0.513071, 0.0290931) -- (0.515157,
  0.0262152) -- (0.516538, 0.0236068) -- (0.517358,
  0.0212476) -- (0.517735, 0.0191172) -- (0.517766,
  0.017196) -- (0.517532, 0.0154651) -- (0.517097,
  0.013907) -- (0.516514, 0.0125052) -- (0.515825,
  0.0112445) -- (0.515065, 0.0101112) -- (0.514262,
  0.00909247) -- (0.513437, 0.00817694) -- (0.512608,
  0.00735417) -- (0.511787, 0.00661477) -- (0.510984,
  0.00595026) -- (0.510207, 0.00535301) -- (0.509462,
  0.00481621) -- (0.508751, 0.00433363) -- (0.508077,
  0.0038998) -- (0.507442, 0.0035097) -- (0.506846,
  0.00315892) -- (0.506287, 0.00284344) -- (0.505767,
  0.00255966) -- (0.505283, 0.00230439) -- (0.504834,
  0.00207471) -- (0.504419, 0.00186806) -- (0.504036,
  0.0016821) -- (0.503682, 0.00151473) -- (0.503357,
  0.00136409) -- (0.503059, 0.0012285) -- (0.502784,
  0.00110643) -- (0.502533, 0.000996534) -- (0.502303,
  0.000897589) -- (0.502093, 0.000808498) -- (0.501901,
  0.000728272) -- (0.501726, 0.000656028) -- (0.501566,
  0.000590967) -- (0.501421, 0.00053237) -- (0.501288,
  0.000479596) -- (0.501168, 0.000432062) -- (0.501058,
  0.000389246) -- (0.500958, 0.00035068) -- (0.500868,
  0.00031594) -- (0.500786, 0.000284646) -- (0.500711,
  0.000256454) -- (0.500643, 0.000231059) -- (0.500582,
  0.000208179) -- (0.500526, 0.000187568) -- (0.500476,
  0.000168998) -- (0.50043, 0.000152268) -- (0.500389,
  0.000137195) -- (0.500351, 0.000123616) -- (0.500318,
  0.00011138) -- (0.500287, 0.000100357) -- (0.500259,
  0.0000904237) -- (0.500234, 0.0000814751) -- (0.500211,
  0.0000734114) -- (0.500191, 0.0000661468) -- (0.500172,
  0.0000596008) -- (0.500156, 0.0000537032) -- (0.50014,
  0.0000483891) -- (0.500127, 0.0000436009) -- (0.500114,
  0.0000392867) -- (0.500103, 0.0000353993) -- (0.500093,
  0.0000318976) -- (0.500084, 0.0000287421) -- (0.500076,
  0.0000258996) -- (0.500068, 0.0000233383) -- (0.500062,
  0.0000210288) -- (0.500056, 0.0000189481) -- (0.50005,
  0.0000170734) -- (0.500045, 0.0000153838) -- (0.500041,
  0.0000138619) -- (0.500037, 0.0000124905) -- (0.500033,
  0.0000112548);

     \draw[latex-](0.354701,  0.079918) -- (0.313381, 0.0851971);

\draw[latex-] (-1.0, 0.58) -- (-1.0, 0.5339);
 \draw[latex-] (1.0, 0.5339) -- (1.0, 0.58026);
 \draw[latex-] (0.7, 0.0) -- (0.71, 0.0);
 \draw[latex-] (-0.71, 0.0) -- (-0.70, 0.0);
 \draw[latex-] (0.01, 1.0) -- (-0.01, 1.0);
 \draw[latex-] (0.30, 0.0) -- (0.29, 0.0);

 \fill[color=gray!110!black] (-1.0,0.0) circle (0.015);
 \fill[color=gray!110!black] (-0.5,0.0) circle (0.015);
 \fill[color=gray!110!black] (0.5,0.0) circle (0.015);
 \fill[color=gray!110!black] (1.0,0.0) circle (0.015);
 \fill[color=gray!110!black] (1.0,1.0) circle (0.015);
 \fill[color=gray!110!black] (-1.0,1.0) circle (0.015);

\end{tikzpicture}
\caption{Phase portrait of $g$.\label{campoginicial}}
\end{figure}
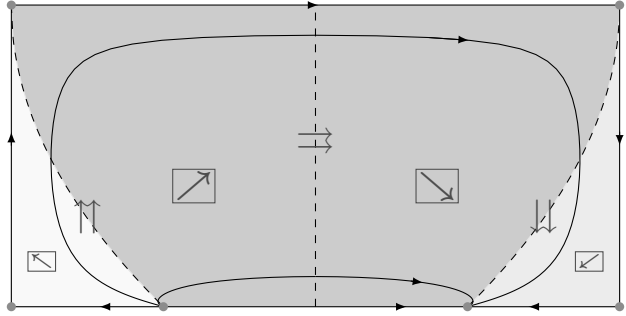

As it happens, this completes the phase portrait because in fact
all interior orbits cross the $y$-axis. To prove this we must
discard the existence of full orbits in the region to the right of
the isocline $x=1-(1-y)^2/2$ or, equivalently (because of the
symmetry properties of the vector field) in the region to the left
of the isocline $x=-1+(1-y)^2/2$. This follows from the fact that
the flow, near $(1,1)$, is equivalent to that associated to
$x'=x^2$, $y'=y$ near $\0$ (this is a consequence of
\cite[Theorem~2.19, pp. 74--75]{DLA}). Alternatively, one can
prove there are no full orbits to the right of $x=1-(1-y)^2/2$ in
a direct way as follows. It clearly suffices to show that the
vector field crosses from left to right all lines $y=1+a(x-1)$,
$a>0$, in the square $(1/2,1)\times (1/2,1)$, that is, $a
g_1(1-t,1-at)-g_2(1-t,1-at)>0$ whenever $0<t<1/2$ and $0<at<1/2$,
when we mean $g=(g_1,g_2)$. Since
\begin{eqnarray*}
 \frac{a g_1(1-t,1-at)-g_2(1-t,1-at)}{at} &=&
   1 + 3 t - a t - 4 t^2 + a t^2 - 2 a^2 t^2 \\
   &&\qquad + (1 + a^2) t^3 +\frac{a^4 t^4}{2} - \frac{a^4 t^5}{4} \\
  &>&  1 + 3 t -\frac{1}{2}  - 2 t + a t^2 -\frac{1}{2} \\
   &&\qquad +(1+a^2) t^3 +\frac{a^4 t^4}{2} - \frac{t}{64} \\
  &=& \frac{63 t}{64} + a t^2 + (1 + a^2) t^3 + \frac{a^4 t^4}{2}
    \; > \; 0,
\end{eqnarray*}
we are done.

Let $\kappa(x)$ be a non-negative $C^1$-function vanishing at
points $x=-1+i/(2r)$, $0\leq i\leq r$ (or at the whole interval
$[-1,-1/2]$ if $r=0$), at points $x=-1/2+i/(s-r)$, $0\leq i\leq
s-r$ (or at the whole interval $[-1/2,1/2]$ if $r=s$), at points
$x=1/2+i/(2j-2s)$, $0\leq i\leq j-s$ (or at the whole interval
$[1/2,1]$ if $s=j$), and at no other points. Then we define
$g_{r,s,j}^+(x,y)=(\kappa(x)+y^2)(1-x^2)g(x,y)$. In this way, we
add some new singular points at the $x$-axis, and all points from
both vertical borders of the rectangle become singular as well,
yet the inner orbits remain the same. Finally we put
$g_{r,s,j}^-(x,y)=-g_{r,s,j}(x,y)$, getting the same phase
portrait with reversed time directions. Some examples of the phase
portraits of these vector fields are shown in Figure~\ref{campog}.

\begin{figure}
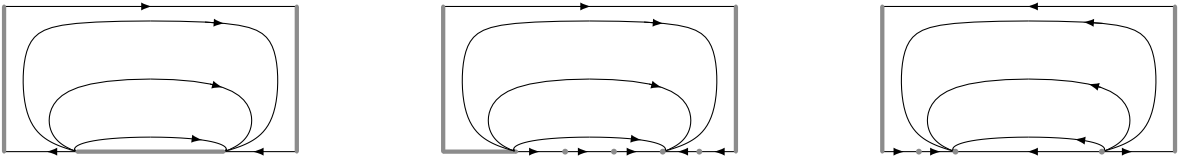

\centering

\caption{Phase portraits of $g^+_{1,1,2}$ (left), $g^+_{0,3,5}$
(center) and $g^-_{2,3,4}$ (right). \label{campog}}
\end{figure}

Let $L$ be a feasible set with base $V$. We are ready to explain
how to construct a polynomial flow $\Phi$ whose associated
feasible set, after fixing the counterclockwise orientation and
choosing an appropriate heteroclinic separatrix of $\Phi$, is
exactly $L$.

Let $n$ be the length of the largest sequence in $V$ and recall
that $n\geq 2$. Also, let $t=\lambda(\emptyset)\geq 1$. Firstly,
we define a vector field $F$ on $\RR$ by gluing (after appropriate
translations and dilatations) some vectors fields
$f_{r,j},g^+_{r,s,j},g^-_{r,s,j}$ (and the null vector field) as
prescribed by $L$.

To begin with, if $(i)\in V$, then we glue at the semi-band
$[i-1,i]\times [0,\infty)$ the vector field
$f_{\sigma(i),\lambda(i)}$ (better to say,
$f_{\sigma(i),\lambda(i)}(2x-2i+1,y)$). Note that the way we
defined the maps $f_{s,j}$ ensures that adjacent pieces glue well
at the orbits $\Upsilon_i:=\{i\}\times [0,\infty)$.

Now, the maximal compact intervals $I$ in $I_i:=[i-i,i]$ such that
$\Inte I\times \{0\}$ contains no singular points will be denoted,
from left to right, by $I_{i,1},\ldots,I_{i,\lambda(i)}$, the flow
travelling to the right on $\Upsilon_{i,k}:=I_{i,k}\times \{0\}$
if and only if $k\leq \sigma(i)$. Certainly, maximal compact
intervals $N$ with $N\times \{0\}$ just consisting of singular
points may exist; we call each of them a $0$-level null interval.

After $F$ has been defined on $[0,t]\times [0,\infty)$, we define
it in $[0,t]\times [-1,0)$. In the rectangles $N\times [-1,0)$,
where $N$ is a $0$-level null interval, we just define $F$ as
zero; and at the rectangles $I_{i,k}\times [-1,0)$ we glue either
the vector field $g^+_{\rho(i,k),\sigma(i,k),\lambda(i,k)}$ (more
properly,
$$
g^+_{\rho(i,k),\sigma(i,k),\lambda(i,k)}((2x-a-b)/(b-a),y+1)
$$
with $I_{i,k}=[a,b]$) or the vector field
$g^-_{\rho(i,k),\sigma(i,k),\lambda(i,k)}$ according to whether
the flow in $\Upsilon_{i,k}$ goes to the right or to the left.
Similarly as before, the maximal compact intervals $I$ in
$I_{i,k}$ such that $\Inte I\times \{-1\}$ contains no singular
points will be denoted, ordered from left to right,
$I_{i,k,1},\ldots,I_{i,k,\lambda(i,k)}$ (write also
$\Upsilon_{i,k,k'}=I_{i,k,k'}\times \{-1\})$, and the flows
travels on $\Upsilon_{i,k,k'}$ in the same direction as in
$\Upsilon_{i,k}$ if and only if $\rho(i,k)<k'\leq \sigma(i,k)$.
Any maximal compact interval $N$ such that $N\times \{-1\}$
consists of singular points will be called a $1$-level null
interval.

Proceeding in this way, we associate inductively  to each vector
$v\in V$ of length $m\geq 2$ an interval $I_v\subset [0,t]$ (and
the corresponding orbit $\Upsilon_v=I_v\times \{-m+2\}$), and
define the $m$-level null intervals. Then  we define $F$ as zero
in $N\times [-m+1,-m+2)$ if $N$ is $m$-level null, or as
$g^+_{\rho(v),\sigma(v),\lambda(v)}$ or
$g^-_{\rho(v),\sigma(v),\lambda(v)}$ in $I_v\times [-m+1,-m+2)$
according to the direction of the flow on $\Upsilon_v$. Note that
the full lowest segment $[0,t]\times \{-n+1\}$ is null, that is,
all its points are singular.

Thus we have completed the definition of $F$ on $[0,t]\times
[-n+1,\infty)$. Note that the map so defined is not locally
Lipschitz (or even continuous) at the orbits $\Upsilon_v$; this
can be easily arranged by multiplying $F$ by  appropriate positive
$C^1$-functions $\tau_v(x)$ in the corresponding semi-open
rectangles $ \Inte I_v\times [-m+1,-m+2)$. We keep calling $F$
this modified map; note that, even so, it needs not be continuous
at the singular points. To conclude the definition of $F$, we
extend it periodically to the whole semi-plane $\R\times
[-n+1,\infty)$ (that is $F(x,y)=F(x+kt,y)$ for any integer $k$)
and define it as zero otherwise.

Before proceeding further, some additional notation must be given.
First, let $\Upsilon_i'=\{i-1/2\}\times [0,\infty)$, $i=1,\ldots,
t$. Also, for any $v\in V$ with length $m\geq 2$, let
$\Upsilon_v'$ be the orbit in $I_v\times (-m+1,-m+2)$
corresponding, after translation and dilatation, to the orbit of
the vector field $g(x,y)$ passing through the point $(0,1/2)$. Now
it is easy to construct a poligonal arc $A$ with endpoints
$(0,1/2)$ and $(t,1/2)$, consisting of alternate horizontal and
vertical segments, so that:
\begin{itemize}
 \item horizontal segments are of type
  $J\times \{-m + \epsilon_{J}\}$
  for some compact interval $J$, some $0<\epsilon_J< 1$
    and $0\leq m<n$;
 \item any two such intervals $J,J'$ have at most one common point,
  and the union of all intervals $J$ is $[0,t]$;
 \item $A$ intersects each orbit $\Upsilon_i,\Upsilon_i'$ at
 exactly one point, and all other orbits $\Upsilon_v,\Upsilon_v'$
 at exactly two points.
\end{itemize}
Observe that the bijection mapping $L$ to the set of these
intersection points that preserves orders (hence mapping
$(t,\lambda(t)+1)$ to $(t,1/2)$), also preserves orbits as
previously meant, that is, every vector $(i,h)$ is mapped  either
to $A\cap \Upsilon_i$ or to $A\cap \Upsilon_i'$ and every pair of
vectors $(v,h),(v,h')$ with $h+h'$ an integer is  mapped either to
$A\cup \Upsilon_v$ or to $A\cup \Upsilon_v'$.

Figure~\ref{BuildingF} illustrates the former construction
starting from the feasible set $L$ described in Table~\ref{tabla}.
The dotted line indicates the arc $A$.

\begin{table}
 \renewcommand{\arraystretch}{1.5}
 \centering
\begin{tabu}{c|[1pt]c}
 $V$ & $L$ \\
 \tabucline[1pt]{-}
 $(1)$ & $(1,\frac{5}{3})$, $(1,4)$ \\
 \hline
 $(1,1)$ & $(1,1,0)$, $(1,1,\frac{1}{3})$, $(1,1,\frac{2}{3})$, $(1,1,1)$\\
 \hline
 $(1,2)$ & $(1,2,0)$, $(1,2,\frac{1}{3})$, $(1,2,\frac{2}{3})$, $(1,2,1)$\\
 \hline
 $(1,3)$ & $(1,3,0)$, $(1,3,\frac{1}{3})$, $(1,3,\frac{2}{3})$, $(1,3,1)$\\
 \hline
 $(2)$ & $(2,\frac{2}{3})$, $(2,2)$ \\
 \hline
 $(2,1)$ & $(2,1,0)$, $(2,1,\frac{1}{3})$, $(2,1,\frac{8}{3})$, $(2,1,3)$\\
 \hline
 $(2,1,1)$ & $(2,1,1,0)$, $(2,1,1,\frac{1}{3})$, $(2,1,1,\frac{2}{3})$, $(2,1,1,1)$\\
 \hline
 $(2,1,2)$ & $(2,1,2,0)$, $(2,1,2,\frac{1}{3})$, $(2,1,2,\frac{2}{3})$, $(2,1,2,1)$\\
 \hline
 $(3)$ & $(3,\frac{2}{3})$, $(3,1)$ \\
 \hline
 $(4)$ & $(4,\frac{5}{3})$, $(4,2)$ \\
 \hline
 $(4,1)$ & $(4,1,0)$, $(4,1,\frac{4}{3})$, $(4,1,\frac{5}{3})$, $(4,1,3)$\\
 \hline
 $(4,1,1)$ & $(4,1,1,0)$, $(4,1,1,\frac{1}{3})$, $(4,1,1,\frac{2}{3})$, $(4,1,1,1)$\\
 \hline
 $(4,1,2)$ & $(4,1,2,0)$, $(4,1,2,\frac{1}{3})$, $(4,1,2,\frac{2}{3})$, $(4,1,2,1)$\\
 \hline
\end{tabu}
\caption{\label{tabla} The elements of the feasible set $L$ and
its base $V$ from Figure~\ref{BuildingF}.}
\end{table}

\begin{figure}
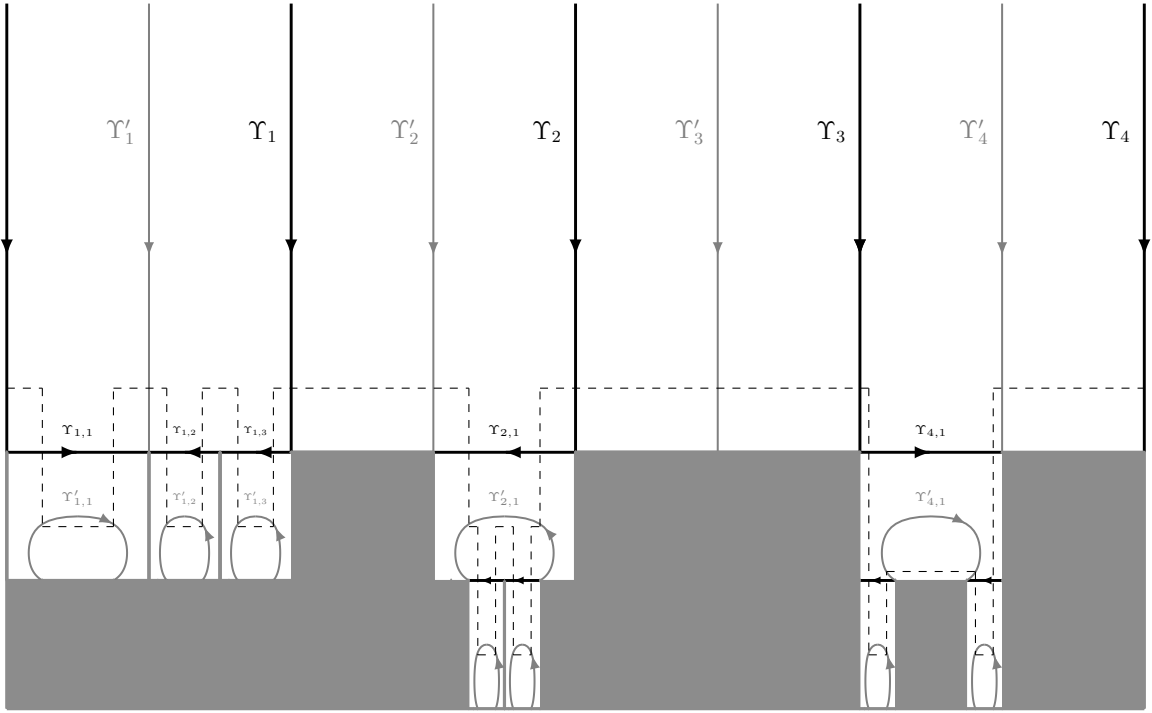

\centering


\caption{Constructing a polynomial flow from a feasible set.
\label{BuildingF}}
\end{figure}

Let $\Xi:\RR\to\RR\setminus \{\0\}$ be given by
$\Xi(r,\theta)=e^{r+\ii 2\pi\theta/t}$. Although $F$ may not be
continuous, the set $T$ of singular points of $F$ is closed and
$F$ is locally Lipschitz in the region $O=\RR\setminus T$; hence,
when restricted to $O$, it has an associated local flow which can
be naturally carried to the region $U=\Xi(O)$ via $\Xi$: call
$\Psi'$ this projected local flow on $U$. Let $\Psi$ be a flow on
$\RR_\infty$ with the same orbits and time orientations as
$\Psi'$, and having singular points outside $U$, that is, at
$K=\Xi(T)\cup \{\0\}$ and $\infty$. This flow induces in
$\mathcal{Q}=\RR_\infty/\sim_K$, in the natural way, a flow
$\Psi_{\sim_K}$ with two singular points, $K$ (now an element of
$\mathcal{Q}$) and $\infty$. Moreover, since $\RR_\infty\setminus
K$ is connected, there is a homeomorphism $H:\mathcal{Q}\to
\RR_\infty$ (Lemma~\ref{cociente}), when we can assume $H(K)=\0$,
$H(\infty)=\infty$. After carrying $\Psi_{\sim_K}$ to $\RR_\infty$
via $H$, we get a flow $\Phi'$ on $\RR_\infty$ having (when
restricted to $\RR$) $\0$ as its global attractor, its separatrix
skeleton consisting of $\0$ and the curves $(H\circ
\Xi)(\Upsilon_v),(H\circ \Xi)(\Upsilon'_v)$, $v\in V$. Using
$C=(H\circ \Xi)(A)$, now a circle around $\0$, choosing an
appropriate orientation $\Theta$ in $\RR$, and taking
$\Sigma=(H\circ \Xi)(\Upsilon_t)$ (recall also
Lemma~\ref{redundante}), we get that $L$ is the canonical feasible
set associated to $\Phi'$, $\Theta$ and $\Sigma$. Composing $H$ if
necessary with a reversing order homeomorphism, we can in fact get
$\Theta$ to be the counterclockwise orientation.

We are almost done. Indeed, since $\Phi'$ has finitely many
unstable orbits, two singular points (the only possible
$\alpha$-limit and $\omega$-limit sets of the flow) and no
periodic orbits, \cite[Lemma~4.1]{JP} (essentially, a corollary of
the main results in \cite{Gu} and \cite{SS}) implies that it is
topologically equivalent to the associated flow to a polynomial
vector field in $\mathbb{S}^2$ and then, as explained in
Section~\ref{preliminary}, to a polynomial flow in $\RR$.
Figure~\ref{BuiltF} shows the resultant flow after collapsing the
flow from Figure~\ref{BuildingF}.

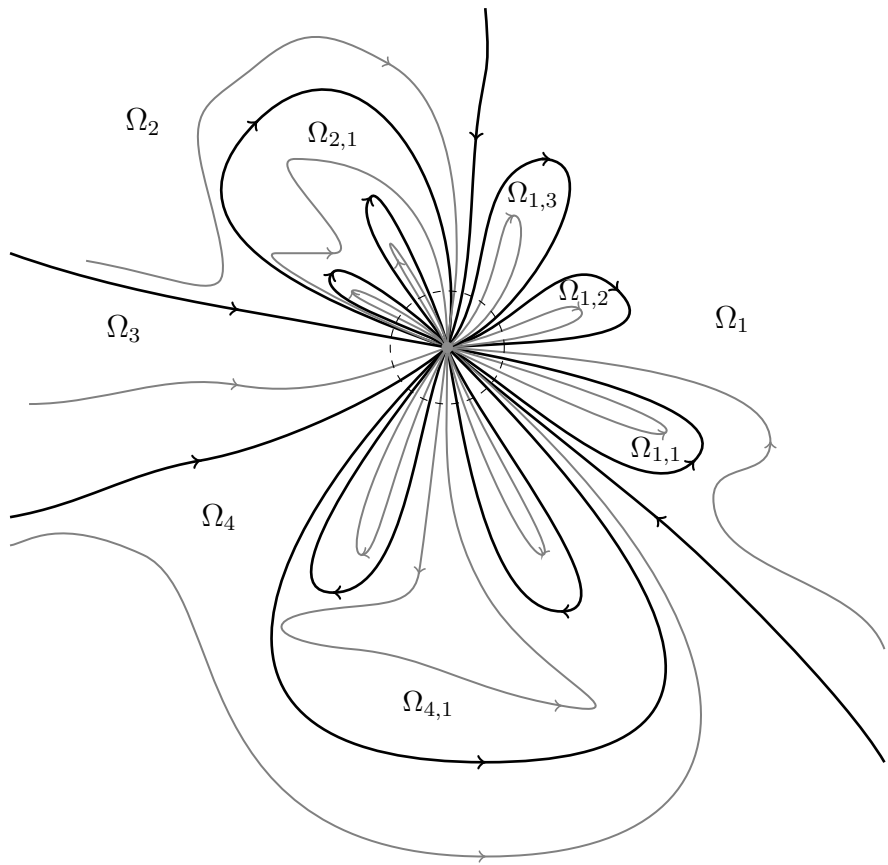
\begin{figure}
 \centering
\begin{tikzpicture}[scale=0.25]


\draw[line width=1]
   (46,5) .. controls +(120:4) and +(320:4) .. (34,18)
           .. controls +(140:4) and +(320:4) ..  (23,27);
\draw[->, line width=1]
    (34,18) .. controls +(140:0.05) and +(320:0.05) ..  (34,18);

\draw[line width=1]
   (25,45) .. controls +(275:4) and +(85:4) .. (24.5,38)
           .. controls +(265:4) and +(70:4) ..  (23,27);
\draw[->, line width=1]
    (24.5,38) .. controls +(265:0.05) and +(85:0.05) ..  (24.5,38);

\draw[line width=1]
   (23,27) .. controls +(330:4) and +(225:4) .. (36,21)
           .. controls +(45:4) and +(350:4) ..  (23,27);
\draw[->, line width=1]
    (36,21) .. controls +(45:0.05) and +(225:0.05) ..  (36,21);

\draw[line width=0.75, color=gray]
   (23,27) .. controls +(335:4) and +(225:1) .. (34.5,22.5)
           .. controls +(45:1) and +(345:4) ..  (23,27);
\draw[->, line width=0.75, color=gray]
   (34.5,22.5) .. controls +(25:0.05) and +(205:0.05) .. (34.5,22.5);

\node at (34,21.5) {\small{$\Omega_{1,1}$}};

\draw[line width=1]
     (23,27) .. controls +(13:4) and +(135:4) ..  (32,30)
           .. controls +(315:4) and +(5:4) ..  (23,27);

\draw[->, line width=1]
    (32,30) .. controls +(325:0.05) and +(135:0.05) ..  (32,30);

\draw[line width=0.75, color=gray]
     (23,27) .. controls +(13:4) and +(135:1) ..  (30,29)
           .. controls +(315:1) and +(5:4) ..  (23,27);
\draw[->, line width=0.75, color=gray]
    (30,29) .. controls +(335:0.05) and +(145:0.05) ..  (30,29);

\node at (30.2,29.8) {\small{$\Omega_{1,2}$}};

\draw[line width=1]
     (23,27) .. controls +(50:4) and +(190:4) ..  (28,37)
           .. controls +(359:4) and +(15:4) ..  (23,27);

\draw[>-, line width=1]
     (28,37) .. controls +(0:0.1) and +(180:0.1) ..  (28,37);

\draw[line width=0.75, color=gray]
     (23,27) .. controls +(42:4) and +(190:1) ..  (26.5,34)
           .. controls +(359:1) and +(25:4) ..  (23,27);
\draw[->, line width=0.75, color=gray]
     (26.5,34) .. controls +(25:0.1) and +(205:0.1) ..  (26.5,34);

\node at (27.5,35) {\small{$\Omega_{1,3}$}};

\draw[line width=0.75, color=gray]
   (46,11) .. controls +(110:4) and +(270:4) .. (37,19)
          .. controls +(90:2) and +(270:2) ..  (40,22)
           .. controls +(90:4) and +(-5:4) ..  (23,27);
\draw[->, line width=0.75, color=gray]
    (40,22) .. controls +(90:0.05) and +(270:0.05) ..  (40,22);



\draw[line width=1]
     (23,27) .. controls +(160:8) and +(225:8) ..  (13,39)
           .. controls +(45:8) and +(80:8) ..  (23,27);
\draw[->, line width=1]
     (13,39) .. controls +(45:0.1) and +(225:0.1) ..  (13,39);

\draw[line width=0.75, color=gray]
   (23,27) .. controls +(160:4.5) and +(180:1.5) .. (14,32)
                    .. controls +(0:2) and +(180:1.5) .. (17,32)
           .. controls +(0:2) and +(180:2) ..  (15,37)
           .. controls +(0:4) and +(90:8) ..  (23,27);

\draw[->, line width=0.75, color=gray]
    (17,32) .. controls +(0:0.05) and +(180:0.05) ..  (17,32);

\node at (17,38.2) {\small{$\Omega_{2,1}$}};

\draw[line width=1]
   (23,27) .. controls +(130:1) and +(225:2) .. (17,31)
           .. controls +(45:1) and +(115:1) ..  (23,27);
\draw[->, line width=1]
    (17,31) .. controls +(60:0.05) and +(240:0.05) ..  (17,31);

\draw[line width=0.75, color=gray]
   (23,27) .. controls +(123:1) and +(225:0.5) .. (18,30)
           .. controls +(45:0.5) and +(121:1) ..  (23,27);
\draw[->, line width=0.75, color=gray]
    (18,30) .. controls +(110:0.05) and +(290:0.05) ..  (18,30);

\draw[line width=1]
   (23,27) .. controls +(110:1) and +(225:2) .. (19,35)
           .. controls +(45:1) and +(90:1) ..  (23,27);
\draw[->, line width=1]
    (19,35) .. controls +(60:0.05) and +(240:0.05) ..  (19,35);

\draw[line width=0.75, color=gray]
   (23,27) .. controls +(110:1) and +(225:0.3) .. (20,32.5)
           .. controls +(45:0.3) and +(90:1) ..  (23,27);
\draw[->, line width=0.75, color=gray]
    (20.4,31.6) .. controls +(110:0.05) and +(290:0.05) ..  (20.4,31.6);

\draw[line width=1]
   (0,32) .. controls +(340:4) and +(170:4) .. (12,29)
           .. controls +(350:4) and +(170:4) ..  (23,27);
\draw[->, line width=1]
    (12,29) .. controls +(350:0.05) and +(170:0.05) ..  (12,29);

\draw[line width=0.75, color=gray]
   (4,31.6) .. controls +(-10:4) and +(225:1) .. (11,30.5)
	        .. controls +(45:1) and +(-80:2) ..  (10,37)
					.. controls +(100:3) and +(220:3) ..  (13,42) 	
					.. controls +(40:3) and +(150:4) ..  (20,42)
					.. controls +(-30:4) and +(90:2) ..  (23.5,32)
					.. controls +(270:2) and +(85:2) ..  (23,27);
	
\draw[->, line width=0.75, color=gray]
    (20,42) .. controls +(-30:0.05) and +(150:0.05) ..  (20,42);



\draw[line width=0.75, color=gray]
   (1,24) .. controls +(0:4) and +(170:4) .. (12,25)
           .. controls +(350:4) and +(200:4) ..  (23,27);
\draw[->, line width=0.75, color=gray]
    (12,25) .. controls +(350:0.05) and +(170:0.05) ..  (12,25);

\draw[line width=1]
   (0,18) .. controls +(10:4) and +(190:4) .. (10,21)
           .. controls +(370:4) and +(220:4) ..  (23,27);
\draw[->, line width=1]
    (10,21) .. controls +(370:0.05) and +(190:0.05) ..  (10,21);




\draw[line width=0.75, color=gray]
   (0,16.5) .. controls +(10:1) and +(155:5) .. (7,16)
            .. controls +(-25:5) and +(180:16) .. (25,0)
           .. controls +(0:16) and +(320:22) ..  (23,27);
\draw[->, line width=0.75, color=gray]
    (25,0) .. controls +(0:0.05) and +(180:0.05) ..  (25,0);

\draw[line width=1]
   (23,27) .. controls +(225:16) and +(180:16) .. (25,5)
           .. controls +(0:16) and +(315:16) ..  (23,27);
\draw[->, line width=1]
    (25,5) .. controls +(0:0.05) and +(180:0.05) ..  (25,5);

\draw[line width=1]
   (23,27) .. controls +(230:4) and +(180:4) .. (17,14)
           .. controls +(0:4) and +(245:4) ..  (23,27);
\draw[->, line width=1]
    (17,14) .. controls +(180:0.05) and +(0:0.05) ..  (17,14);


\draw[line width=0.75, color=gray]
   (23,27) .. controls +(233:6) and +(180:1) .. (18.5,16)
           .. controls +(0:1) and +(242:6) ..  (23,27);
\draw[->, line width=0.75, color=gray]
    (18.5,16) .. controls +(200:0.05) and +(20:0.05) ..  (18.5,16);

\node at (22,8) {\small{$\Omega_{4,1}$}};

\draw[line width=1]
   (23,27) .. controls +(280:4) and +(180:4) .. (29,13)
           .. controls +(0:4) and +(290:4) ..  (23,27);
\draw[->, line width=1]
    (29,13) .. controls +(180:0.05) and +(0:0.05) ..  (29,13);

\draw[line width=0.75, color=gray]
   (23,27) .. controls +(282:4) and +(180:1) .. (28,16)
           .. controls +(0:1) and +(288:4) ..  (23,27);

\draw[->, line width=0.75, color=gray]
    (28,16) .. controls +(250:0.05) and +(70:0.05) ..  (28,16);

\draw[line width=0.75, color=gray]
   (23,27) .. controls +(260:1) and +(80:2) .. (21.5,15)
           .. controls +(260:2) and +(10:3) ..  (16,13)
                     .. controls +(190:3) and +(175:4) ..  (18,11)
                     .. controls +(355:4) and +(170:5) .. (29,8)
                     .. controls +(350:5)  and +(300:5) .. (25,14)
                    .. controls +(120:5) and +(270:5) .. (23,27);
\draw[->, line width=0.75, color=gray]
    (29,8) .. controls +(0:0.05) and +(180:0.05) ..  (29,8);
\draw[->, line width=0.75, color=gray]
    (21.5,15) .. controls +(260:0.05) and +(80:0.05) ..  (21.5,15);


\fill[color=gray!110!black] (23,27) circle (.3);

\draw[dashed, line width=0.4] (23,27) circle [radius=3];


\node at (38,28.5) {$\Omega_1$}; \node at (7,39) {$\Omega_2$};
\node at (6,28) {$\Omega_3$}; \node at (11,18) {$\Omega_4$};

\end{tikzpicture}
\caption{\label{BuiltF} The phase portrait of the flow labelled by
the feasible set from Table~\ref{tabla}.}
\end{figure}

\begin{remark}
 \label{daigual}
 Since any flow having $\0$ as a global
 attractor and finitely many separatrices is
 topologically equivalent to a polynomial
 flow, and polynomial flows have the finite sectorial decomposition
 property, we get that finiteness of separatrices and sectors are,
 in fact, equivalent properties in this setting (compare to
 Remark~\ref{FSDP}).
\end{remark}

\section{Proof of Theorem~\ref{explicito}}
 \label{example}

To study the nature of the phase portrait of  \eqref{eqex} near
$\0$ and at the infinity one could use, in principle,
desigularization \cite[Chapter~3]{DLA} and the Poincar\'e
compactification \cite[Chapter~5]{DLA}. In the present case this
leads, however, to very heavy calculations; thus the need to rely
on specific (yet elementary) arguments, as those given below.

Since the polynomial $(1+x^2)^2+x^3$ has no real zeros,  $\0$ is
the only singular point of the associated local flow to
\eqref{eqex}. The isocline corresponding to the horizontal
direction of the vector field is the union of the curves $y=0$ and
$y^2+x^3=0$. Thus, the $x$-axis consists of $\0$ and two regular
orbits (both going to $\0$ in positive time) and there are no
periodic orbits, as they should enclose the singular point. The
isocline corresponding to the vertical direction of the vector
field is the curve  $(1+x^2)y + x^3=0$. Finally, the isoclines
divide the plane in six regions $U_i$, $1\leq i\leq 6$, where the
flow has a well-defined direction: see
Table~\ref{tabladirecciones} and Figure~\ref{exphase}.

\begin{table}
 \renewcommand{\arraystretch}{1.5}
 \centering
\begin{tabu}{c|[1pt]c}
 Directions & Regions \\
 \tabucline[1pt]{-}
 $x'<0$, $y'>0$ & $U_1=\{(x,y): y>0, y^2 + x^3>0\}$ \\
 \hline
 $x'<0$, $y'<0$ & $U_2=\{(x,y): y^2 + x^3<0, (1+x^2)y + x^3>0\}$ \\
 \hline
 $x'>0$, $y'<0$ & $U_3=\{(x,y): (1+x^2)y + x^3<0, y>0\}$\\
 \hline
 $x'>0$, $y'<0$ & $U_4=\{(x,y): y<0, y^2 + x^3<0\}$\\
 \hline
 $x'>0$, $y'>0$ & $U_5=\{(x,y): y^2 + x^3>0,(1+x^2)y + x^3<0\}$ \\
 \hline
 $x'<0$, $y'>0$ & $U_6=\{(x,y): (1+x^2)y + x^3>0, y <0\}$\\
 \hline
\end{tabu}
\caption{\label{tabladirecciones} Directions of the vector field
for the system \eqref{eqex}.}
\end{table}

\begin{figure}
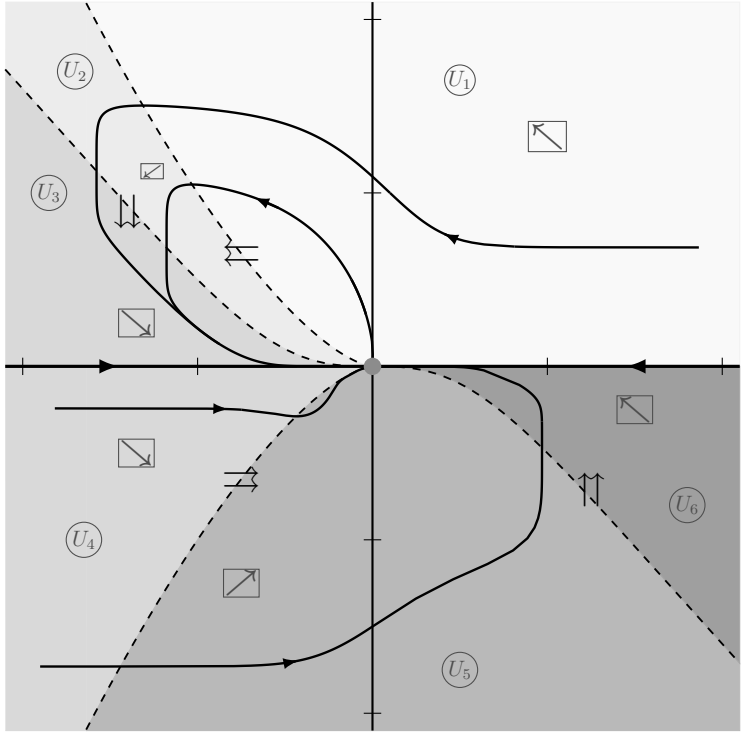

\centering

\caption{Phase portrait of $x'=-((1+x^2)y + x^3)^5$,
  $y'=y^2(y^2+x^3)$. \label{exphase}}
\end{figure}

\medskip \noindent\emph{Claim 1: The origin is a global attractor
of \eqref{eqex}.}

First of all, observe that orbits starting in $U_2$ go to $U_3$,
and orbits starting in $U_3$ go to $\0$. Similarly, orbits
starting in $U_4$ go to $U_5$, orbits starting in $U_5$ either go
to $\0$ or to $U_6$, and orbits starting in $U_6$ go to $\0$. As a
consequence, in order to prove the claim, it is enough to show
that any orbit starting in $U_1$  meets the curve $y^2+x^3=0$.

Let $P(x,y)=-((1+x^2)y + x^3)^5$ and $Q(x,y)=y^2(y^2+x^3)$ be the
components of the vector field and put $U_1'=U_1\cap \{(x,y):y\geq
1\}$. Then we have
\begin{equation}
  \label{aplanado}
 -1\leq \frac{Q(x,y)}{P(x,y)} \leq 0 \qquad \text{for any $(x,y)\in U_1'$}
\end{equation}
because if $x\geq 0$, then
 $$
 Q(x,y)=y^4+y^2x^3\leq (1+x^2)^5y^5+5(1+x^2)^4y^4x^3\leq |P(x,y)|,
 $$
while if $x\leq 0$, we use that $y\geq -x$ holds in $U_1'$ to get
 $$
 Q(x,y)\leq y^4\leq y^5\leq (y+yx^2+x^3)^5=|P(x,y)|.
 $$
Now, realize that if an orbit starts in $U_1$, then either it
crosses $y^2+x^3=0$, or goes to $U_1'$. Therefore, to prove the
claim, it suffices to show that if $(x_0,y_0) \in U_1'$, then the
orbit (corresponding to the solution) $(x(t),y(t))$ of
\eqref{eqex} starting at $x(0)=x_0$ and $y(0)=y_0$ meets
$y^2+x^3=0$. But, due to \eqref{aplanado}, we have $y'(t) \leq -
x'(t)$ and then $y(t)\leq x_0+y_0 -x(t)$ whenever the orbit stay
in $U_1'$. In other words, the orbit lies below the line
$y=x_0+y_0-x$ while staying in $U_1'$. Since this line intersects
$y^2+x^3=0$, Claim~1 follows.

\medskip \noindent\emph{Claim 2: The origin is not positively
stable for \eqref{eqex}.}

Given any $y_0 > 0$, let $(x(t),y(t))$ be the orbit of
\eqref{eqex} starting at $x(0)=0$ and $y(0)=y_0$. According to
Claim~1, this orbit must travel to $U_2$, then to $U_3$, and
finally converge to $\0$. In particular, it  meets the line
$y=-2x$. Let $t_*$ be the (smallest) positive time for which
$y(t_*)=-2 x(t_*)$ and define $Y(y_0)=y(t_*)$.

To prove the claim, it suffices to show that $Y(y_0) > 1/4$ (this
bound is very conservative; numerical estimations suggest that the
optimal bound is approximately $0.831$). We proceed by
contradiction assuming $Y(y_0) \leq  1/4$. Then $-1/8 \leq x(t)
\leq 0$ for any $0\leq t \leq t_*$.

For the sake of clarity, in this paragraph we assume $0\leq t \leq
t_*$ and shorten $x(t)$ as $x$ and $y(t)$ as $y$. Since $x\leq 0$,
we trivially have
\begin{equation}
 \label{factor1}
 y + \frac{x^3}{1 + x^2} \leq  y+ (-x)^{3/2}.
\end{equation}
We assert that
\begin{equation}
 \label{factor2}
 y + \frac{x^3}{1 + x^2} \leq 2\left(y-(-x)^{3/2}\right)
\end{equation}
is true as well. Observe that \eqref{factor2} is equivalent to
 $$
 2(1+x^2)(-x)^{3/2} + x^3 \leq (1+x^2)y
 $$
and, taking into account that $y \geq - 2 x$, a sufficient
condition for this to happen is
 $$
 (-2 x (1 + x^2) - x^3)^2 - (2 (1 + x^2) (-x)^{3/2})^2 \geq 0,
 $$
which is true indeed:
\begin{eqnarray*}
 (-2 x (1 + x^2) - x^3)^2 - (2 (1 + x^2) (-x)^{3/2})^2
 &=&  x^2 (4 + 4 x + 12 x^2 + 8 x^3 + 9 x^4 + 4 x^5) \\
&\geq& 4 x^2(1 + x + 2 x^3 +   x^5) \\
&\geq& 4 x^2\left(1 - \frac{1}{8} - \frac{1}{256}
-\frac{1}{32768}\right)\geq 0.
\end{eqnarray*}
Finally, we have
\begin{equation}
 \label{factor3}
 \frac{1}{(1+x^2)^5} \geq
 \frac{1}{(1+1/64)^5}>\frac{1}{2}.
\end{equation}
Putting together \eqref{factor1}, \eqref{factor2} and
\eqref{factor3}, we get
 $$
 \frac{Q(x,y)}{P(x,y)} =
 -\frac{y^2(y+(-x)^{3/2})(y-(-x)^{3/2})}{(1+x^2)^5(y+x^3/(1+x^2))^5}
 \leq  -\frac{1}{4 y}.
 $$

As a consequence, for every $0\leq t \leq t_*$, we have $2 y'(t)
y(t) \geq - x'(t)/2$ and therefore
 $$
 y(t)^2 \geq y_0^2 -x(t)/2 > - x(t)/2,
 $$
that is, the orbit lies over the parabola $y^2=-x/2$. Since this
parabola intersects $y=-2x$ at the point $(-1/8,1/4)$, we obtain
the desired contradiction $Y(t_0)>1/4$, and Claim~2 follows.

\medskip \noindent\emph{Claim 3: The origin is an elliptic saddle
for \eqref{eqex}.}

Let $R$ be the union set of all heteroclinic orbits of
\eqref{eqex}, that is, the closed lower half-plane (except $\0$)
and all orbits intersecting the positive semi-$y$-axis. By
Claims~1 and 2, $R$ is a radial region strictly included in
$\RR\setminus \{\0\}$ (Proposition~\ref{atractorestable}).
Moreover, it is clear that this flow does not allow a pair of
incomparable homoclinic orbits. Then $\Bd R=\Gamma\cup \{0\}$,
$\Gamma$ being the only regular homoclinic separatrix of the flow
(the other separatrices are the positive semi-$x$-axis and $\0$),
and $\0$ is an elliptic saddle (Remark~\ref{ejemplofacil}).

\medskip

Claims~1, 2 and 3 complete the proof of
Theorem~\ref{explicito}.

\section*{Acknowledgements}
We are indebted to Professor Armengol Gasull (Universitat
Aut\`onoma de Barcelona), who brought this problem to the second
author's attention.

This work has been partially supported by Ministerio de
Econom\'{\i}a y Competitividad, Spain, grant MTM2014-52920-P. The
first author is also supported by Fundaci\'on S\'eneca by means of
the program ``Contratos Predoctorales de Formaci\'on del Personal
Investigador'', grant 18910/FPI/13.

\bigskip
\noindent \emph{J. G. Esp\'{\i}n Buend\'{\i}a and V. Jim\'enez
L\'opez's address:} {\sc Departamento de Matem\'a\-ticas,
Universidad de Murcia, Campus de Espinardo, 30100 Murcia, Spain.}

e-mails: {\tt josegines.espin@um.es, vjimenez@um.es}

\end{document}